\documentclass[11pt]{article}
\usepackage{latexsym,color,amsmath,amsthm,amssymb,amscd,amsfonts}
\usepackage{comment}
\newtheorem{lemma}{Lemma}[section]

\newtheorem{remark}[lemma]{Remark}

\newtheorem*{remark*}{Remark}

\def\R{{\mathbb R}}

\def\upddots{\mathinner{\mkern 1mu\raise 1pt \hbox{.}\mkern 2mu
\mkern 2mu \raise 4pt\hbox{.}\mkern 1mu \raise 7pt\vbox {\kern 7
pt\hbox{.}}} }

\newcommand{\spn}{{Sp_{2n}(\F)}}

\newcommand{\gspn}{{GSp_{2n}(\F)}}

\newcommand{\mspn}{{\widetilde{Sp_{2n}(\F)}}}

\newcommand{\slt}{{SL_{2}(F)}}
\newcommand{\mslt}{\widetilde{{SL_{2}(F)}}}

\newcommand{\F}{F}
\newcommand{\1}{ \bold 1}

\newcommand{\Of}{\mathbb O_{\F}}
\newcommand{\Pf}{\mathbb P_{\F}}

\newcommand{\N}{\mathbb N}
\newcommand{\Z}{\mathbb Z}

\newcommand{\half}{\frac{1}{2}}

\newcommand{\ab} {|\!|}

\newcommand{\C}{\mathbb C}

\def\>{\rangle}
\def\<{\langle}

\newtheorem{lem}{Lemma}[section]
\newtheorem{thm}{Theorem}[section]

\newtheorem{prop}{Proposition}

\numberwithin{equation}{section}

\newcommand{\msl}{\widetilde{SL_2(\F)}}

\newcommand{\pbm}{\widetilde{B^d}}

\def\dotunion{
\def\dotunionD{\bigcup\kern-9pt\cdot\kern5pt}
\def\dotunionT{\bigcup\kern-7.5pt\cdot\kern3.5pt}
\mathop{\mathchoice{\dotunionD}{\dotunionT}{}{}}} \date{}

\begin{document}
\author{David Goldberg and Dani Szpruch}
\title {Plancherel measures for coverings of p-adic $SL_2(F)$ } \maketitle
\begin{abstract} In these notes we compute the Plancherel measures associated with genuine principal series representations of $n$-fold covers of $p-$adic $SL_2$. Along the way we also compute a higher dimensional metaplectic analog of Shahidi local coefficients. Our method involves new functional equations utilizing the Tate $\gamma$-factor and a metaplectic counterpart. As an application we prove an irreducibility theorem.
\end{abstract}
\section{Introduction}

Let $G$ be quasi-split reductive group defined over a $p$-adic field. The Plancherel measure is an analytic invariant associated with a parabolic induction on $G$. It was conjectured by Langlands, \cite{Lang},  that this invariant is a ratio of certain $L$-functions and hence of arithmetic significance. Under some mild assumptions, this conjecture was proven for generic inducing data by Shahidi in \cite{Sha 90} where the Plancherel measure was utilized to derive some results in harmonic analysis on $p$-adic reductive groups. Since, up to a well understood positive constant, the Plancherel measure equals an inverse of a scalar arising from a composition of two standard intertwining integrals, Shahidi was able to use his theory of local coefficients for that proof. In this paper we follow Shahidi's approach and compute the Plancherel measure for coverings of $SL_2(F)$ even though uniqueness of Whittaker model fails.

Fix an integer $n \geq 1$. Let $F$ be a p-adic which contains the full group of $n^{\operatorname{th}}$ order roots of 1. We shall also assume that its residual characteristic is prime to $n$.  Let $\widetilde{G(F)}$ be the $n$-fold cover of $G(F)=SL_2(F)$ afforded by the Kubota cocycle, see \cite{Kub}. Let $H(F)$ be the diagonal subgroup inside $G(F)$, let $N(F)$ be the group of upper triangular unipotent matrices in $G(F)$ and let $B(F)=H(F) \ltimes N(F)$. $\widetilde{H(F)}$,  the inverse image in $\widetilde{G(F)}$ of $H(F)$ is not Abelian. Rather, it is a two step nilpotent group. In Section \ref{car rep} we construct the irreducible genuine admissible representations of $\widetilde{H(F)}$ using an analog of Kazhdan-Patterson's standard maximal Abelian group, \cite{KP}. It is interesting to note the similarity to the $n=2$ case, if $n$ is even then the construction involves the Weil index of a character of second degree. This phenomena is not present in the construction in \cite{KP}. Let $\tau$ be a genuine smooth admissible irreducible representation of $\widetilde{H(F)}$. Let $I(\tau,s)=\operatorname{Ind}^{\widetilde{G(F)}}_{\widetilde{B(F)}} \tau_s$ be the representation of $\widetilde{G(F)}$ parabolically induced from $\tau_s$. Here $s$ is a complex parameter. Let $$A_w(\tau,s): I(\tau,s) \rightarrow I(\tau^w,-s)$$ be the standard intertwining integral associated with the unique non-trivial Weyl element of $G(F)$. Then $A_w(\tau,s)$ induces a map
$$\widehat{A}_w(\tau,s): Wh_{\psi}(\tau^w,-s)  \rightarrow  Wh_{\psi}(\tau,s) $$
where $\psi$ is a non-trivial character of $N(F)$ and  $Wh_{\psi}(\tau,s) $ is the space of $\psi$-Whittaker functionals on  $I(\tau,s)$. Unless $n \leq 2$, $Wh_{\psi}(\tau,s)$ is not one dimensional. In fact, $$\dim \bigl( Wh_{\psi}(\tau,s) \bigr)=\frac {n}{gcd(n,2)}.$$

In Section \ref{meta sha} we fix certain bases for $Wh_{\psi}(\tau^w,-s)$ and $Wh_{\psi}(\tau,s)$ and compute the matrix $D(\tau,s)$ representing $\widehat{A}_w(\tau,s)$ with respect to these bases. Using $D(\tau,s)$ and $D(\tau^w,-s)$ we prove in Section \ref{irr res} our main result: if we take the measure in the integral defining the intertwining operators to be the self dual measure with respect to $\psi$ then

\begin{equation} \label{Plancherel as L} A_{w^{-1}} (\tau^w,-s) \circ A_{w}(\tau,s)=q^{e(\chi^n,\psi)}\frac{L \bigl(ns,\chi^n \bigr)L \bigl(-ns,\chi^{-n}\bigr)}{L \bigl(1-ns,\chi^{-n} \bigr)L \bigl(1+ns,\chi^{n}\bigr)}\operatorname{Id} \end{equation}

Here $\chi$ is a character of $F^*$, $\chi^n$ is an invariant of $\tau$ and $q^{e(\chi^n,\psi)}$ is a positive constant defined in Section \ref{basics}. Observe that while $D(\tau,s)$ and $D(\tau^w,-s)$ depend on $\psi$, the constant in the right hand side of \eqref{Plancherel as L} is independent of $\psi$ except for the choice of measure. Moreover, if $\tau$ and $\psi$ are unramified then $q^{e(\chi^n,\psi)}=1$ and \eqref{Plancherel as L} may be deduced from the metaplectic Gindikin-Karpilevich formula. Indeed, if $I(\tau,s)$ is unramified and $\phi^0_{\tau,s}$ is its normalized spherical vector then from Theorem 12.1 of \cite{Mc2} it follows that
$$ A_{w}(\tau,s)\bigl(\phi^0_{\tau,s})=\frac{L \bigl(ns,\chi^n \bigr)}{L \bigl(1+ns,\chi^{n}\bigr)}\phi^0_{\tau^w,-s}.$$

As an immediate application of \eqref{Plancherel as L} we find all the reducible genuine principal series representations of $\widetilde{G(F)}$ induced from a unitary data, see Theorem \ref{savin}. The matrix $D(\tau,s)$ is a higher dimensional metaplectic analog of (the inverse of)  Shahidi local coefficients. In Lemma 1.33 of \cite{KP}, Kazhdan and Patterson have computed this matrix in the context of $n$-fold covers of $GL_r(F)$. Their result involved Gauss sums. However, while the computation in \cite{KP} addresses only representations induced from unramified data, our computation applies for all inducing data. Moreover, our method deviates from those used in \cite{KP}. In the cases where $n$ is odd we have used the Tate $\gamma$-factor and functional equation in a fundamental way: in Section \ref{ari and meta ari} we follow an elegant argument used by Ariturk for the $n=3$ case and prove a functional equation suitable for our computation.

In the cases where $n$ is even we prove a similar functional equation involving the metaplectic $\widetilde{\gamma}$-factor arising from the functional equation proven in \cite{Sz 3}. The role of $\widetilde{\gamma}$ seems to be as natural and important here as the role of the Tate $\gamma$-factor. We note that $\widetilde{\gamma}$ is represented by a certain Tate-type integral. This integral initially appeared in unpublished notes of W. Jay Sweet, \cite{Sweet}, where it was computed and utilised for the study of degenerate principal series representations of the metaplectic double cover of $\spn$.  It then appeared in \cite{Sz 2} which contains the $n=2$ case of the computation presented here. We expect $\widetilde{\gamma}$ to play a similar roll in the representation  theory of even fold covers of classical groups. In \cite{Sweet}, Sweet thanks Kudla for a helpful suggestion about the computation of $\widetilde{\gamma}$. To the best of our knowledge, this computation is not available in print. In an appendix we correct this situation and reproduce the results given in \cite{Sweet}.

The main motivation for this work is an ongoing project \cite{FGS} in which we study the analog of  Kazhdan-Patterson's exceptional representations for coverings of classical and similitude groups. The technique presented here is sufficient for the completion of the project, namely, for producing an analog of  Lemma 1.33 of \cite{KP} for all the groups in discussion. Since our method is applicable for coverings of $GL_r(F)$, our result can be used to identify the Gauss sums in \cite{KP} as $\epsilon$-factors.

Although we shall not develop this point here, our computations show that even-fold covers are more delicate than odd-fold covers as the analytic properties of $D(\tau,s)$ depend on the choice of Whittaker character for these groups. This was already noted in the $n=2$ case, see \cite{Sz 2}. This phenomena is responsible for fact that the irreducible modules that appear in the restriction of the unramified distinguished representation studied by Gelbart and Piatetski-Shapiro in \cite{GP80} and \cite{GP81} to the double cover of $\slt$ have a Whittaker model with respect to exactly one orbit of Whittaker characters.  We expect this phenomena to reappear in the study of even-fold covers of classical groups.

In a recent paper, \cite{GanGao}, Gan and Gao raised the question whether the Langlands-Shahidi method could be extended to metaplectic groups other than the metaplectic double cover of $\spn$. We hope that the results contained in these notes will contribute to this discussion.

We would like to thank Solomon Friedberg for helpful discussions on the subject matter and to Wee Teck Gan and Freydoon Shahidi for their valuable comments. We would also like to thank Gordan Savin for making his notes available to us.

\section{Preparations}
\subsection{Basic notations} \label{basics}
Let $\F$ be a p-adic field. Denote by $p$ its residual characteristic and by $q$ the cardinality of its residue field. Denote by $\Of$ its ring of integers. Fix $\varpi$, a generator of $\Pf$, the maximal ideal of $\Of$. We normalize the absolute value on $F$ such that $\ab \varpi \ab =q^{-1}$. For $\psi$, a non-trivial character of $F$ and $\chi$, a character of $\F^*$, we define $e(\psi)$ and $e(\chi)$ to be the conductors of $\psi$ and $\chi$ respectively and we define
$$e(\psi,\chi)=e(\psi)-e(\chi).$$
Let $S(F)$ be the space of Schwartz functions on $\F^*$. Given $\phi \in S(F)$ we define $\widehat{\phi} \in S(F)$ to be its $\psi$ Fourier transform, i.e.,
$$\widehat{\phi}(x)=\int_F \phi(y) \psi(xy) \, d_\psi y.$$
Here $d_\psi y$ is the self dual measure with respect to $\psi$. We define $d_\psi^*x=\ab x \ab^{-1} d_\psi x$. It is a Haar measure on $F^*$. If $\psi$ is unramified we write $dx$ and $d^*x$ for
$d_\psi x$ and  $d^*_\psi x$ respectively. For $a\in F^*$ we denote by $\psi_a$ the character of $\F$ given by $x \mapsto \psi(xa)$. Let $\gamma_F(\psi)$ be the unnormalized Weil index of character of second degree, see \cite{Weil}, defined by
$$\lim_{r \rightarrow \infty} \int_{\Pf^{-r}} \psi(x^2) \, d_\psi x.$$ It known that $\gamma^8_F(\psi)=1$. If $\F$ is of odd residual characteristic and $\psi$ is spherical then $\gamma_F(\psi)=1$. Let $\gamma_\psi:F^* \mapsto \C$ be the normalized Weil index defined by

$$\gamma_\psi(a)=\frac {\gamma_F(\psi_a)}{\gamma_F(\psi)}.$$
Recall  that $\gamma_\psi({\F^*}^2)=1$ and that for $x,y \in F^*$ we have
$$\gamma_\psi(xy)=\gamma_\psi(x)\gamma_\psi(x)(x,y)_2,$$
where $(\cdot,\cdot)_2$ is the quadratic Hilbert symbol.
\begin{remark} \label{remark on split} If $\F$ is of odd residual characteristic and $-1 \in {F^*}^2$ then $\gamma_\psi(\F^*) \in \mu_2$. If, in addition, $\psi$ is unramified then $\gamma_\psi(\Of^*)=1$ and for exactly one representative $u$ of $\Of^* / {\Of^*}^2$ we have $\gamma_\psi(u\varpi)=1$.
\end{remark}

\subsection{Tate and metaplectic-Tate gamma factors.} \label{T and meta T}
Let $\chi$ be a character of $\F^*$, let $\psi$ be a non-trivial character of $\F$ and let $s$ be a complex parameter . For  $\phi \in S(F)$ we let $\zeta(s,\chi,\phi)$ be its Mellin transform, namely, the meromorphic continuation of
$$\int_{F^*} \phi(x) \ab x \ab^s \chi(x) \, d_\psi^*x.$$
Let $$\gamma(s,\chi,\psi)=\epsilon(s,\chi,\psi)\frac{L(1-s,\chi^{-1})}{L(s,\chi)}$$
be the Tate $\gamma$-factor, \cite{T}. It is defined via the functional equation
\begin{equation} \label{tate def} \zeta(1-s,\chi^{-1},\widehat{\phi})=\gamma(s,\chi,\psi)\zeta(s,\chi,\phi). \end{equation}
Recall that both Mellin transforms in the equation above are given by absolutely convergent integrals in some common vertical strip. The following are stated in \cite{T2}.

\begin{equation} \label{Tate gamma} {\epsilon}(1-s,\chi^{-1},\psi)=\chi(-1)\epsilon^{-1}(s,\chi,\psi), \end{equation}
\begin{equation} \label{epsilon twist} \epsilon(s+t,\chi,\psi)=q^{e(\psi,\chi)t}\epsilon(s,\chi,\psi),\end{equation}
\begin{equation} \label{epsilon change psi} \epsilon(s,\chi,\psi_a)=\chi(a)\ab a \ab^{s-\half}\epsilon(s,\chi,\psi). \end{equation}
Observe that \eqref{Tate gamma} and \eqref{epsilon twist} imply that
\begin{equation} \label{epsilon product} \epsilon(s,\chi,\psi)\epsilon(-s,\chi^{-1},\psi)=\chi(-1)q^{-e(\psi,\chi)}.\end{equation}

For $\phi \in S(F)$ we now define $\widetilde{\phi}:\F^* \rightarrow \C$ by

$$\widetilde{\phi}(x)=\int_{F^*} \phi(y)\gamma_\psi^{-1}(xy) \psi(xy) d_\psi y.$$
Although $\widetilde{\phi}(x)$ is typically not an element of $S(\F)$ it was proven in \cite{Sz 3} that
$$\int_{F^*} \widetilde{\phi}(x) \chi(x) \ab x \ab^{s} d^*_\psi x$$
converges absolutely for $a<Re(s)<a+1$, for some  $a \in \R$, to some rational function in $q^{-s}$ (it was assumed in \cite{Sz 3} that $\psi$ is spherical but this is unnecessary). This enables the natural definition of $\zeta(s,\chi,\widetilde{\phi})$. Moreover, $\zeta(1-s,\chi^{-1},\widetilde{\phi})$ and $\zeta(s,\chi,\psi)$ are both given by absolutely convergent integrals in some
common vertical strip and there exists a function $\widetilde{\gamma}(s,\chi,\psi)$ such that
$$\zeta(1-s,\chi^{-1},\widetilde{\phi})=\zeta(s,\chi,\phi)\widetilde{\gamma}(s,\chi,\psi).$$
For all $\phi \in S(F)$. It was also proven in \cite{Sz 3} that $\widetilde{\gamma}(\chi^{-1},1-s,\psi)$ is the meromorphic continuation of
$$\lim_{r \rightarrow \infty}\int_{\Pf^{-r}}\chi(x) \ab x \ab^s \gamma_\psi(x)^{-1} \psi(x) \, d_\psi^* x.$$
The computation of this integral is contained in unpublished notes of W. Jay Sweet, \cite{Sweet}, and is provided here in the appendix. We have
\begin{equation} \label{meta gama formula} \widetilde{\gamma}(1-s,\chi^{-1},\psi)=\gamma_F^{-1}(\psi_{-1}) \chi(-1)\gamma^{-1}(2s,\chi^{2},{\psi_{_2}}) \gamma(s+\half,\chi,\psi). \end{equation}
We note here that although the arguments in \cite{Sz 3} are correct, the formula given there for $\widetilde{\gamma}(s,\chi,\psi)$ is slightly mistaken. The formula given here is the correct one.

\subsection{$n^{th}$ power Hilbert symbol}
Fix an integer $n \geq 1$. We shall assume that $\F^*$ contains the full group of $n^{th}$ roots of unity. Denote this cyclic group by $\mu_n$. We identify $\mu_n$ with the group of $n^{th}$ roots of unity in $\C^*$ and suppress this identification. Let $$( \cdot, \cdot):F^* \times F^* \rightarrow \mu_n$$ be the $n^{th}$ power Hilbert symbol. It is a bilinear form on $\F^*$  that defines a non-degenerate bilinear form on $\F^* / {F^*}^n$. It is known that for all $x,y \in F^*$
$$(x,y)(y,x)=(x,-x)=1.$$
Recall also that for all $x \in F^*$, $$(x,x)=(-1,x) \in \mu_2.$$

For $r \in \N$ let $\F^*_r$ be the subgroup of $\F^*$ of elements whose valuation lies in $r\Z$. Clearly
$\F^*_r \simeq \Of^* \times r\Z$.

\begin{lem} \label{kernal p n 1}If $gcd(p,n)=1$ then  $$\{x\in \F^* \mid (x,y)=1 \, \forall y\in \F^*_n \}=\F^*_n.$$
\end{lem}
\begin{proof}  Section 5 of Chapter XIII of \cite{Weil book}.
\end{proof}
For $x \in F^*$ define $\eta_x$ to be the character of $F^*$ given by
$$\eta_x(y)=(x,y).$$
\begin{lem} Suppose $gcd(p,n)=1$. Then $\eta_x$ is trivial on $1+\Pf$. Furthermore, $\eta_x$ is unramified if and only if $x\in F^*_n$.
\end{lem}
\begin{proof} By Hensel's lemma, $1+\Pf \subseteq {\F^*}^n$ provided that $gcd(p,n)=1$. This proves the first assertion. The second is clear.
\end{proof}
From this point we assume that  $gcd(p,n)=1$. This assumption is relaxed in the appendix and in remark \ref{last remark}. We shall denote $$d=\frac{n} {gcd(n,2)}.$$ Define $$\beta_\varpi:F_d^* \rightarrow F^*$$ by $$\beta_\varpi(u\varpi^{md})=u\varpi^m.$$ It is an isomorphism. Note that $\beta_\varpi$ depends on the choice of a uniformizer.
\begin{lem} Suppose that $n$ is even.  For $x,y \in F_d^*$ we have $(x,y)^2=1$ and
$$(x, y)=\bigr(\beta_\varpi(x),\beta_\varpi(y) \bigl)_2$$
\end{lem}
\begin{proof}
Given $x,y \in F^*$ write  $x=u\varpi^{dm}, \, y=u'\varpi^{dm'}$ where $u,u' \in \Of^*$ we have
$$(x,y)=(u,\varpi)^{dm'}(u',\varpi)^{-dm}(-1,\varpi)^{d^2mm'}.$$
Since $n=2d$ the first assertion follow. Denote by $\chi_o$ the unique non-trivial quadratic character of $\Of^*$. With this notation we have
$$(x,y)=\chi^{m'}_{o}(u)\chi^{m}_{o}(u')(-1,\varpi)^{d^2mm'}$$
and
$$\bigr(\beta_\varpi(x),\beta_\varpi(y) \bigl)_2=\chi^{m'}_{o}(u)\chi^{m}_{o}(u')(-1,\varpi)_2^{mm'}.$$
Thus, the proof is done once we show that
$$(-1,\varpi)^{d^2mm'}=(-1,\varpi)_2^{mm'}.$$
Indeed, if $n=0 \, (\operatorname{mod }4)$ then $d$ is even and since $F$ contains a primitive $4^{\operatorname{th}}$ root of 1 we have $-1 \in {F^*}^2$. This shows that
$$(-1,\varpi)^{d}=(-1,\varpi)_2=1.$$
Suppose now that $n=2 \, (\operatorname{mod }4)$. It is sufficient to show that $$(-1,\varpi)=(-1,\varpi)_{2}.$$
We note that
$$ (-1,\varpi)_{_2}=\begin{cases} 1 & -1 \in {F^*}^2; \\ -1 & -1 \not \in {F^*}^2,\end{cases} \, \, \, (-1,\varpi)=\begin{cases} 1 & -1 \in {F^*}^n; \\ -1 & -1 \not \in {F^*}^n. \end{cases}$$
One can easily see that the assertion that $-1 \in {F^*}^2$ is equivalent to the assertion that $-1 \in {F^*}^n$ (in fact, both assertions are equivalent to the assertion that $F^*$ contains the full group of $2n^{\operatorname{th}}$ roots of 1).
\end{proof}

\begin{lem} \label{split hilbert} Define $\xi_{\psi,\varpi}:F^*_d \rightarrow \C^1$ to be the trivial map if n is odd and $\gamma_\psi^{-1} \circ \beta_\varpi$ if $n$ is even. Then $\xi_\psi$ splits the Hilbert symbol on $F_d^* \times F_d^*$, i.e.,
$$\xi_{\psi,\varpi}(xy)=\xi_{\psi,\varpi} (x)\xi_{\psi,\varpi} (y)(x,y)$$
for all $x,y \in \F^*_d$.\\
\end{lem}
\begin{proof} Clear.
\end{proof}
\begin{lem} \label{split porp}Suppose that $n$ is even.\\
1. If $d$ is odd then $\xi_{\psi,\varpi}=\gamma_\psi^{-1}$. \\
2. If $d$ is even then
$$\xi_{\psi,\varpi}(x)= \begin{cases} \gamma_\psi^{-1}(x)  & x \in F^*_n; \\ \gamma_\psi^{-1}(x\varpi)  & x  \not \in F^*_n,  \end{cases}=\gamma_\psi^{-1}(x)\begin{cases} 1  & x \in F^*_n; \\ \gamma_\psi^{-1}(\varpi)(\varpi,x)^d  & x  \not \in F^*_n.  \end{cases}$$
\end{lem}
\begin{proof} Part 1 follows from the fact that if $n=2 \, ( \operatorname{mod} \, 4)$ then  $\beta_\varpi(y) \in y{F^*}^2$. Part two follows from the fact that if $n=0 \, ( \operatorname{mod}\, 4)$ then
for $y \in F^*_n$ we have $\beta_\varpi(y) \in y{F^*}^2$ while for $y  \not \in F^*_n$ we have $\beta_\varpi(y) \in \varpi y{F^*}^2$. 
\end{proof}
\subsection{Characteristic functions}
Denote by $F^*_{n,k}$ the set of elements of $\F^*$ of the form $u\varpi^m$ where $u \in \Of^*$ and $m\in k+n\Z$. Thus, $F^*_n=F^*_{n,0}$.
Fix a unit $u_{_0}$ such that $\xi=(u_{_0},\varpi)$ is a primitive $n^{th}$ root of 1. Define a function  $$\beta_{n,k}:F^* \rightarrow \C$$ by
$$\beta_{n,k}(x)=\frac 1 n \sum_{l=0}^{n-1}(u_{_0},x\varpi^{-k})^l.$$
\begin{lem} $\beta_{n,k}$ is the characteristic function of  $F^*_{n,k}$.
\end{lem}
\begin{proof}
Clearly, if $x \in F^*_{n,k}$ then $\beta_k(x)=1$. On the other hand, if $x=u \varpi^m$ where $m \not \in k+n\Z$ then $1 \neq (u_{_0},x\varpi^{-k}) \in \mu_n.$ Thus,  $\beta_k(x)=0$.
\end{proof}
Since $x\mapsto (u_{_0},x)$ is an unramified character then $$(u_{_0},x)=\ab x \ab^c$$ for some purely imaginary number $c$. Thus, $\xi=q^{-c}$ and
\begin{equation} \label{beta as sum for inv} \beta_{n,k}(x)=\frac 1 n \sum_{l=0}^{n-1} \xi^{-kl}\ab x \ab^{lc}=\frac 1 n \sum_{l=0}^{n-1} \xi^{kl}\ab x \ab^{-lc}. \end{equation}

\section{Functional equations.} \label{ari and meta ari}
In this Section we generalize an argument that appears  in \cite{Ariturk}. Precisely, Lemmas 3.1 and 3.2 in \cite{Ariturk} are the $n=3$ case of  Theorem \ref{few functional equations} below.
For $\phi \in S(F)$ define $$\zeta_{n,k}(s,\chi,\phi)=\zeta(s,\chi,\phi \cdot \beta_{n,k}).$$

\begin{thm} \label{few functional equations} For $0\leq k<n$ we have

\begin{equation} \label{aritur} \zeta_{n,k}(s,\chi,\widehat{\phi})=\sum_{m=0}^{n-1}\theta_{m}(s,\chi,\psi) \zeta_{n,m+e(\psi,\chi)-k}(1-s,\chi^{-1},\phi), \end{equation}
where, for unramified $\chi$,
$$\theta_{m}(s,\chi,\psi) =\epsilon^{-1}(s,\chi,\psi)L(ns,\chi^n) \bigl(\chi(\varpi)q^{-s}\bigr)^m \times
\begin{cases} (1-q^{-1})  & 0 \leq m \leq n-2; \\ \\ L^{-1}(1-ns,\chi^{-n}) & m=n-1, \end{cases}$$
and where for ramified $\chi$,
$$\theta_{m}(s,\chi,\psi) =\begin{cases} \chi(-1)\epsilon^{-1}(s,\chi,\psi)  & m=0; \\ \\ 0 & m \neq 0. \end{cases}$$
\end{thm}
\begin{proof} It is sufficient to prove this result for $s$ such that all Mellin transforms in \eqref{aritur} are given by absolutely convergent integrals. We have
$$\zeta_{n,k}(s,\chi,\widehat{\phi})=\int_{F^*} \beta_{n,k}(x)\widehat{\phi}(x) \chi(x) \ab x \ab^s d^*_\psi x=\frac 1 n \sum_{l=0}^{n-1} \xi^{-kl}  \int_{F^*} \widehat{\phi}(x) \chi(x) \ab x \ab^{s+lc} d^*_\psi x=$$
$$\frac 1 n \sum_{l=0}^{n-1} \xi^{-kl}  \gamma\bigl(1-(s+lc),\chi^{-1},\psi \bigr) \int_{F^*}\phi(x) \chi^{-1}(x) \ab x \ab^{1-s-lc} d^*_\psi x=$$
$$\int_{F^*} \delta_{k,\psi,\chi}(x)\phi(x) \chi^{-1}(x) \ab x \ab^{1-s} d^*_\psi x$$
where $$\delta_{k,\psi,\chi,s}(x)=\sum_{l=0}^{n-1} \xi^{-kl}  \gamma\bigl(1-(s+lc),\chi^{-1},\psi \bigr)\ab x \ab^{-lc}.$$
It remains to show that
$$\delta_{k,\psi,\chi,s}=\sum_{m=0}^{n-1}\theta_{m}(s,\chi,\psi) \beta_{n,m+e(\psi,\chi)-k}.$$
Suppose first that $\chi$ is ramified. In this case, by \eqref{Tate gamma} and \eqref{epsilon twist} we have
$$\delta_{k,\psi,\chi,s}(x)=\chi(-1)\epsilon^{-1}(s,\chi, \psi) \sum_{l=0}^{n-1} \xi^{l\bigl(e(\psi,\chi)-k \bigr)}  \ab x \ab^{-lc}.$$
By \eqref{beta as sum for inv} we are done. Suppose now that $\chi$ is unramified. By  \eqref{Tate gamma} and \eqref{epsilon twist} we now have
$$\delta_{k,\psi,\chi,s}= \epsilon^{-1}(s,\chi, \psi) \frac 1 n \sum_{l=0}^{n-1} \ab x \ab^{-lc}\xi^{l(e(\psi)-k)}  \frac{1-q^{-1}q^{s}\chi^{-1}(\varpi)\xi^{-l}}{1-q^{-s}\chi(\varpi) \xi^l}=$$
$$\prod_{m=0}^{n-1}(1-q^{-s}\chi(\varpi) \xi^m)^{-1}\epsilon^{-1}(s,\chi, \psi) \times $$ $$ \frac 1 n \sum_{l=0}^{n-1} \Bigl(\ab x \ab^{-lc}\xi^{l(e(\psi)-k)} (1-q^{-1}q^{s}\chi^{-1}(\varpi)\xi^{-l})\! \! \! \! \! \prod_{0 \leq m \leq n-1 \, m\neq l} \! \! \! \! \! \bigl(1-q^{-s}\chi(\varpi) \xi^m\bigr) \Bigr).$$
Since $\xi$ is a primitive element in $\mu_n$, we have the elementary identities
$$\prod_{m=0}^{n-1}(1-q^{-s}\chi(\varpi) \xi^m)=\bigl(1-q^{-ns}\chi^n(\varpi)\bigr)$$ and
$$\prod_{0 \leq m \leq n-1 \, m\neq l} \! \! \!  \bigl(1-q^{-s}\chi(\varpi) \xi^m\bigr)=\sum_{m=0}^{n-1} q^{-ms}\chi^m(\varpi) \xi^{lm}.$$
Thus,
$$\delta_{k,\psi,\chi,s}(x)=\chi(-1)\epsilon^{-1}(s,\chi, \psi)L(\chi^n,ns) \times $$
$$\Biggl( \frac{\bigl(q^{-s}\chi(\varpi)\bigr)^{n-1}}{L(\chi^{-n},1-ns)}\frac 1 n \sum_{l=0}^{n-1} \ab x \ab^{-lc}\xi^{l(e(\psi)-k-1)}+(1-q^{-1})\sum_{m=0}^{n-2}(q^{-s}\chi(\varpi))^m \Bigl(\sum_{l=0}^{n-1} \ab x \ab^{-lc}\xi^{l(e(\psi)-k+m)} \Bigr) \Biggl).$$
Using \eqref{beta as sum for inv} once more we complete the proof.
\end{proof}
We now define
$$\zeta_{n,k}(s,\chi,\widetilde{\phi})=\zeta(s,\chi,\widetilde{\phi}\beta_{n,k})$$
and prove a metaplectic analog to Theorem \ref{few functional equations}.
\begin{thm} \label{meta few functional equations} Suppose that $n$ is even. For $0\leq k<n$ we have

\begin{equation} \label{meta ari}\zeta_{n,k}(s,\chi,\widetilde{\phi})=\sum_{m=0}^{n-1}\widetilde{\theta}_{m}(s,\chi,\psi) \zeta_{n,m+e(\psi,\chi^2)-k}(1-s,\chi^{-1},\phi), \end{equation}
where $\widetilde{\theta}_{m}(s,\chi,\psi)$ is defined as follows. If $\chi$ is unramified then
$$\widetilde{\theta}_{m}(s,\chi,\psi) =\gamma_F^{-1}(\psi_{-1}) \epsilon^{-1}(2s,\chi^{2},{\psi_{_2}}) \epsilon(s+\half,\chi,\psi)  \times $$
$$\begin{cases}  q^{-\half}q^s\ \chi^{-1}(\varpi) & m=n-1; \\ \\ (1-q^{-1})\bigl(q^{-s}\chi(\varpi)\bigr)^{m}L(ns,\chi^n)  & m\in 2\Z; \\ \\ 0 & otherwise. \end{cases}$$
If $\chi$ is ramified but $\chi^2$ is unramified then,

$$\widetilde{\theta}_{m}(s,\chi,\psi) =\gamma_F^{-1}(\psi_{-1}) \epsilon^{-1}(2s,\chi^{2},{\psi_{_2}}) \epsilon(s+\half,\chi,\psi) \bigl(q^{-s}\chi(\varpi)\bigr)^{m-1}L(\chi^n,ns) \times$$
$$\begin{cases}  L^{-1}(1-ns,\chi^{-n}) & m=n-1; \\ \\ (1-q^{-1}) & m\in 1+2\Z, \, m \neq n-1; \\ \\ 0 & otherwise. \end{cases}$$
If $\chi^2$ is ramified then,
$$\widetilde{\theta}_{m}(s,\chi,\psi) =\begin{cases} \gamma_F^{-1}(\psi_{-1}) \chi(-1)\epsilon^{-1}(2s,\chi^{2},{\psi_{_2}}) \epsilon(s+\half, \chi,\psi)  & m=0; \\ \\ 0 & m \neq 0. \end{cases}$$
\end{thm}
\begin{proof} It is sufficient again to prove this result for $s$ such that all Mellin transforms in \eqref{meta ari} are given by absolutely convergent integrals. Similar to the proof of Theorem \ref{few functional equations} we have
$$\zeta_{n,k}(s,\chi,\widetilde{\phi})=\int_{F^*} \widetilde{\delta}_{k,\psi,\chi}(x)\phi(x) \chi^{-1}(x) \ab x \ab^{1-s} d^*_\psi x,$$
where $$\widetilde{\delta}_{k,\psi,\chi,s}(x)=\sum_{l=0}^{n-1} \xi^{-kl}  \widetilde{\gamma}\bigl(1-(s+lc),\chi^{-1},\psi \bigr)\ab x \ab^{-lc}.$$
We now continue using \eqref{meta gama formula}. Note that since $n$ is even, $F$ is of odd residual characteristic. This implies that $e(\psi_2)=e(\psi)$. Suppose first that $\chi^2$ is ramified. We claim that this implies that $e(\chi)=e(\chi^2)$. Indeed, if $e(\chi)=1$ then there is nothing to proof. If $e(\chi)>1$ this equality follows from the fact that $1+\Pf \subseteq  {F^*}^2$. Thus,
$$\widetilde{\delta}_{k,\psi,\chi,s}(x)=\gamma_F^{-1}(\psi_{-1}) \chi(-1)\epsilon^{-1}(2s,\chi^{2},{\psi_{_2}}) \epsilon(s+\half,\chi,\psi)\sum_{l=0}^{n-1} \xi^{l\bigl(e(\psi,\chi)-k\bigr)}\ab x \ab^{-lc}.$$
By arguments we used already in the proof of Theorem \ref{aritur} we are done in this case. Suppose now that $\chi$ is unramified. We have
$$\widetilde{\gamma}\bigl(1-(s+lc),\chi^{-1},\psi \bigr)=\gamma_F^{-1}(\psi_{-1}) \epsilon^{-1}(2s,\chi^{2},{\psi_{_2}}) \epsilon(s+\half,\chi,\psi) \xi^{le(\psi)} \frac {(1-q^{-1})+q^{-\half}(q^{s}\chi^{-1}\bigl(\varpi)-q^{-s}\chi(\varpi)\bigr)}{1-q^{-2s}\chi^2(\varpi)\xi^{2l}}.$$
One now repeats similar arguments to those used in the unramified case in Theorem \ref{few functional equations}. In the course of the computation one uses the fact that since $\xi^2$ is a primitive element in $\mu_d$ we have
$$\prod_{m=0}^{n-1}(1-q^{-2s}\chi^2(\varpi) \xi^{2m})=\bigl(1-q^{-ns}\chi^n(\varpi)\bigr)^2$$ and
$$\prod_{0 \leq m \leq n-1 \, m\neq l} \! \! \!  \bigl(1-q^{-2s}\chi^2(\varpi) \xi^m\bigr)=\bigl(1-q^{-ns}\chi^n(\varpi)\bigr)\sum_{m=0}^{d-1} q^{-2ms}\chi^{2m}(\varpi) \xi^{2lm}.$$
This ultimately gives
$$\widetilde{\delta}_{k,\psi,\chi,s}(x)=\gamma_F^{-1}(\psi_{-1}) \chi(-1)\epsilon^{-1}(2s,\chi^{2},{\psi_{_2}}) \epsilon(s+\half,\chi,\psi) \times$$
$$ \Bigl( q^{-\half}q^s\ \chi^{-1}(\varpi)\beta_{n-1-k+e(\psi)}(x)+(1-q^{-1})L(ns,\chi^n)\sum_{m=0}^{d-1} \bigl(q^{-s}\chi(\varpi)\bigr)^{m} \beta_{2m-k+e(\psi)}(x)\Bigr).$$
If $\chi$ is ramified but $\chi^2$ is unramified then
$$\widetilde{\gamma}\bigl(1-(s+lc),\chi^{-1},\psi \bigr)=\gamma_F^{-1}(\psi_{-1}) \chi(-1) \epsilon^{-1}(2s,\chi^{2},{\psi_{_2}}) \epsilon(s+\half,\chi,\psi) \xi^{(\bigl(le(\psi)-1\bigr)} \frac {1-q^{-1}q^{2s}\chi^{-2}(\varpi)\xi^{2l}}{1-q^{-2s}\chi^2(\varpi)\xi^{-2l}}.$$
Repeating the same arguments as above we obtain
$$\widetilde{\delta}_{k,\psi,\chi,s}(x)=\gamma_F^{-1}(\psi_{-1}) \chi(-1)\epsilon^{-1}(2s,\chi^{2},{\psi_{_2}}) \epsilon(s+\half,\chi,\psi)L(ns,\chi^n) \times $$
$$\Bigl( \bigl( q^{-s}\chi(\varpi) \bigr)^{n-2}L^{-1}(1-ns,\chi^{n})\beta_{n,e(\psi)-k+(n-1)}+(1-q^{-1})\sum_{m=0}^{d-1}\bigl( q^{-s}\chi(\varpi) \bigr)^{2m}\beta_{n,e(\psi)-k+(2m-1)} \bigr).$$

\end{proof}

\section{$n$ fold cover of $SL_2(F)$}
\subsection{Construction of the cover}
Let ${SL_2(F)}$ be the group of two by two matrices with entries in $F$ whose determinant is 1. Let $N(F) \simeq F$ be the group of upper triangular unipotent matrices. Let $H(F)\simeq F^*$ be the group of diagonal elements inside  $SL_2(F)$. Denote $B(F)=H(F) \ltimes N(F)$. For $x \in F$, and  $a\in F^*$ we shall write
$$n(x)=\begin{pmatrix} _{1} & _{x}\\_{0} & _{1 }
\end{pmatrix}, \quad  h(a)=\begin{pmatrix} _{a} & _{0}\\_{0} & _{a^{-1} }
\end{pmatrix}, \quad  w=\begin{pmatrix} _{0} & _{1}\\_{-1} & _{0}
\end{pmatrix}.$$
Let $\widetilde{SL_{2}(F)}$ be the topological central extension of $\slt$ by $\mu_n$ constructed using the Kubota cocylce, \cite{Kub}. More precisely, we realize  $\widetilde{SL_2(F)}$ as the set $SL_2(F) \times \mu_n$ along with the multiplication
$$\bigl(g,\epsilon \bigr)\bigl(g',\epsilon \bigr)=\bigl(gg',c(g,g')\epsilon \epsilon'\bigr),$$
where
\begin{equation}\label{rao}c(g_1,g_2)=\bigl(x(g_1g_2)x^{-1}(g_1),x(g_1g_2)x^{-1}(g_2)\bigr).\end{equation} Here
$$x \begin{pmatrix} _{a} & _{b}\\_{c} &
_{d}\end{pmatrix}=\begin{cases} c & c \neq 0; \\ d & c=0.
\end{cases}$$

We shall denote by $s$ the map from  ${SL_{2}(F)}$ to $\widetilde{SL_{2}(F)}$ given by $s(g)=(g,1)$. For a subset $A$ of $SL_{2}(F)$ we shall denote by $\widetilde{A}$ its primage in $\widetilde{SL_{2}(F)}$.
\subsection{Representations of the Cartan subgroup} \label{car rep}
From \eqref{rao} it follows that $$c \bigl(h(a),h(b) \bigr)=(b,a).$$ This implies that $s(h(a))$ and $s(h(b))$ commute if and only if $(b,a)^2=1$. Let $\widetilde{H_0(F)}$ be the center of $\widetilde{H(F)}$. One immediately sees  that
$$\widetilde{H_0(F)}= \widetilde{H^d(F)}.$$
Define now
$$H_d(F)=\{h(a) \mid a\in F^*_d \}.$$

From Lemmas \ref{kernal p n 1} and \ref{split hilbert} it  follows that $\widetilde{H_d(F)}$ is a maximal Abelian subgroup of $\widetilde{H(F)}$. It is an analog to Kazhdan-Patterson's standard maximal Abelian subgroup, see \cite{KP}. Observe that if $n$ is odd then $c \bigl(H_d(F),H_d(F) \bigr)=\{1\}$. This means that $\mslt$ splits over $H_d(F)$ via the trivial section. However, if $n$ is even then $c \bigl(H_d(F),H_d(F) \bigr)=\mu_2$.

A representation of $\mslt$ or any of its subgroups is called genuine if $(I_2,\epsilon)$ acts by $\epsilon$. By Lemma \ref{split hilbert}

$$\bigl(h(a),\epsilon \bigr) \mapsto \epsilon \xi_{\varpi,\psi}(a)$$ is a genuine character of $\widetilde{H_d(F)}$. Since the quotient of two genuine characters of $\widetilde{H_d(F)}$ is a character which factors through the projection to $H_d(F)$ it follows that any genuine character of $\widetilde{H_d(F)}$ is given by
$$\bigl(h(a),\epsilon\bigr) \mapsto \chi_{\varpi,\psi}\bigl(h(a),\epsilon\bigr)=\epsilon \chi(a) \xi_{\varpi,\psi}(a)$$ where $\chi$ is a character of $F^*_d$.

Remark: Suppose that $n$ is even. If $-1 \in {F^*}^2$ then we may think of  $\xi_{\varpi,\psi}$ as a map into $\mu_n$. In this case
$\bigl(h(a),\epsilon \bigr) \mapsto \bigl(h(a),\epsilon \xi_{\varpi,\psi}  \bigr)$ defines an isomorphism from $\widetilde{H_d(F)}$ to $H_d(F) \times \mu_n$. We shall not use this fact since the parametrization given above of the genuine characters of $\widetilde{H_d(F)}$ is sufficient for our purposes.

\begin{lem} \label{cartan rep} Any genuine smooth admissible irreducible representation of $\widetilde{H(F)}$ may be realized as
$$i\bigl(\chi_{\varpi,\psi}\bigr)=\operatorname{Ind}_{\widetilde{H_d(F)}}^{\widetilde{H}} \chi_{\varpi,\psi}.$$
The representations $i\bigl(\chi_{\varpi,\psi}\bigr)$ and $i\bigl(\chi'_{\varpi,\psi}\bigr)$ are isomorphic if and only if
$$\chi'=\chi\eta_\varpi^{2m}$$ for some integer $m$. 
\end{lem}
\begin{proof} By a variation of  Stone-von Neumann Theorem, see Theorem 3 of \cite{Mc}, the genuine smooth admissible irreducible representations of $\widetilde{H(F)}$ are parameterized by genuine characters of  $\widetilde{H_{_0}}$. A realization of  a genuine smooth admissible irreducible representations $\tau$ of $\widetilde{H(F)}$ whose central character is $\chi_{\tau}$ is constructed by first extending $\chi_{\tau}$ to a character of a maximal Abelian subgroup of $\widetilde{H(F)}$ and then by inducing. This proves the first assertion. To prove the second assertion note that any character of $\widetilde{H_{_0}}$ has exactly $[\widetilde{H_d(F)}:\widetilde{H_{_0}}]=d$ extensions to $\widetilde{H_d(F)}$ and that given a genuine character $\chi_{\varpi,\psi}$ of $\widetilde{H_d(F)}$, the set
$$\{(\chi \eta_\varpi^{2m})_{\varpi,\psi} \mid m=0,1,\cdots, d-1 \}$$ consists of $d$ characters of $\widetilde{H_d(F)}$ whose restriction to $\widetilde{H_{_0}}$ are equal. 
\end{proof}

Given a complex parameter $s$ and a genuine character $\chi_{\varpi,\psi}$ of $\widetilde{H_d(F)}$ we define another genuine character  $\bigl(\chi_{\varpi,\psi},s \bigr)$ of $\widetilde{H_d(F)}$
by $$g=(a,\epsilon) \mapsto \chi_{\varpi,\psi}(g)\ab a \ab^s.$$ Note that  $\chi_{\varpi,\psi}=\bigl(\chi_{\varpi,\psi},0 \bigr).$ Similarly, given a genuine smooth admissible irreducible representation $\tau$ of $\widetilde{H(F)}$ whose central character is $\chi_\tau$  and a complex parameter $s$ we define $\tau_s$ to be the
genuine smooth admissible irreducible representation of  $\widetilde{H(F)}$ whose central character is
$$(x,\epsilon)\mapsto  \chi_\varpi(a,\epsilon) \ab a \ab^s.$$
Observe that if $\tau \simeq  i\bigl(\chi_{\varpi,\psi}\bigr)$ then $\tau_s \simeq i\bigl(\chi_{\varpi,\psi},s \bigr)$.

\subsection{Genuine principal series representations and Whittaker functionals} \label{prin whi}
$\mslt$ splits over $N(F)$ via the trivial section and  $\widetilde{H(F)}$ normalizes $s\bigl( N(F) \bigr)$. Therefore, any representation of $\widetilde{H(F)}$
can be extended to a representation of $\widetilde{B(F)}$ by defining it to be trivial on  $s\bigl( N(F) \bigr)$. Similar to the linear case we shall identify the representations of $\widetilde{B(F)}$ and  $\widetilde{H(F)}$. A genuine principal series representation of $\mslt$ is defined to be a representation parabolically induced from a genuine smooth admissible representation of  $\widetilde{B(F)}$. Thus, using induction in stages, any genuine principal series representation of $\mslt$ may be realized as
$$I\bigl(\chi_{\varpi,\psi},s\bigr)=\operatorname{Ind}_{\widetilde{B_d(F)}}^{\mslt} \bigl(\chi_{\varpi,\psi},s \bigr).$$
where $B_d(F)=H_d(F)N(F)$.  We shall assume that parabolic inductions are normalized. 

Let $\psi'$ be (another) non-trivial character $F$. A $\psi'$-Whittaker functional $\lambda$ on  a representation $(\pi,V)$ of $\mslt$ is a functional on $V$ which satisfies
$$\lambda \circ \pi\bigl(s(x)\bigr)=\psi'(x)\lambda.$$
Let $Wh_{\psi'}\bigl(\chi_{\varpi,\psi},s \bigr)$  be the space of $\psi'$ Whittaker functionals on $I\bigl(\chi_{\varpi,\psi},s\bigr)$. Unlike the linear case the dimension of this space is not 1. Arguing exactly as in Lemmas 1.3.1 and 1.3.2 of \cite{KP} we have
\begin{lem} Let $\bigl(\chi_{\varpi,\psi},s \bigr)$ be a genuine character of $\widetilde{H_d(F)}$.\\
1. $\dim Wh_{\psi'}\bigl(\chi_{\varpi,\psi},s \bigr)=[\widetilde{H}:\widetilde{H_d(F)}]=d.$\\ \\
2. If  $\ab \chi\bigl(\varpi^d) \ab <q^{Re(s)d}$ then for any $h \in \widetilde{H}$ and $f \in I\bigl(\chi_{\varpi,\psi},s\bigr)$ the integral
$$\int_F f\bigl(hs(wn(x))\bigr)\psi^{-1}(x) d_{\psi'}x$$ converges absolutely to a polynomial in $q^{-s}$. Moreover, for all $s$,
$$\lim_{r \rightarrow \infty}  \int_{\Pf^{-r}} f\bigl(hs(wn(x))\bigr)\psi^{-1}(x) d_{\psi'}x$$
exists.\\ \\
3. Let $\lambda_{h,\psi',\chi_{\varpi,\psi},s}(f)$ denote the analytic continuation of the integral defined in Part 2. The map $f \mapsto \lambda_{h,\psi',\chi_{\varpi,\psi},s}(f)$ is a $\psi'$ Whittaker functional on $I\bigl(\chi_{\varpi,\psi},s\bigr)$.\\ \\
4. Let $A$ be a set of representatives of $\widetilde{H} / \widetilde{H_d(F)}$. The set
$$\{\lambda_{h,\psi',\chi_{\varpi,\psi},s} \mid h \in A \}$$ is a basis for  $Wh_{\psi'}\bigl(\chi_{\varpi,\psi},s \bigr)$.
\end{lem}
We shall now fix once and for all the set

$$A=\{ s(\varpi^i) \mid i=0,1,\cdots d-1 \}.$$
as a set of representatives of $\widetilde{H} / \widetilde{H_d(F)}$ and we shall write $$\lambda_{i,\psi',\chi_{\varpi,\psi},s}=\lambda_{ s(\varpi^i),\psi',\chi_{\varpi,\psi},s}.$$
\begin{remark} \label{last remark} If we drop the assumption that the residual characteristic of $F$ is prime to $n$,  $\widetilde{H^d(F)}$ is still the center of $\widetilde{H(F)}$ but $\widetilde{H_d(F)}$ is not always Abelian. However, one can verify at once that the inverse image of
$$ \{h(x^d\varpi^k) \mid x\in F^*, \, k \in \Z \}$$ inside $\msl$ is always an Abelian subgroup of $\widetilde{H(F)}$. The index of this  subgroup is $[{\Of^*}^d:\Of^*]$. This index is an upper bound for the dimension of the irreducible genuine smooth admissible representations of $\widetilde{H(F)}$ and for  the dimension of the space of Whittaker functionals on a genuine principal series representations of $\msl$.
\end{remark}
\section{Metaplectic Shahidi local coefficients} \label{meta sha}
\subsection{Definition}
Fix  a genuine character $\chi_{\varpi,\psi}$ of $\widetilde{H_d(F)}$. If $\ab \chi\bigl(\varpi^d) \ab <q^{Re(s)d}$ then for any $f \in I\bigl(\chi_{\varpi,\psi},s\bigr)$ the integral
$$\int_F f\bigl(s(w)s(n(x))g \bigr) d_{\psi'}x$$ converges absolutely to a rational function in $q^{-s}$. We shall denote its meromorphic continuation by  $A_{w}\bigl(\chi_{\varpi,\psi},s\bigr)(f)$. Away from its poles,

$$A_{w}\bigl(\chi_{\varpi,\psi},s\bigr):I(\chi_{\varpi,\psi},s\bigr) \rightarrow I(\chi^{-1}_{\varpi,\psi},-s\bigr)$$ is a well defined $\mslt$ map. Define now

$$\lambda^w_{i,\psi',\chi_{\varpi,\psi},s}=\lambda_{i,\psi',\chi^{-1}_{\varpi,\psi},-s}\circ A_{w}\bigl(\chi_{\varpi,\psi},s\bigr).$$
We have
$$\lambda^w_{i,\psi',\chi_{\varpi,\psi},s}=\sum_{j=0}^d \tau(i,j,\chi_{_{\varpi,\psi}},s,\psi')\lambda_{j,\psi',\chi_{\varpi,\psi},s},$$
where the functions $\tau(i,j,\chi_{_{\varpi,\psi}},s,\psi')$ are rational in $q^{-s}$.  Define now $D(\chi_{_{\varpi,\psi}},s,\psi')$ to be the $d \times d$ matrix  whose entries are $\tau(i,j,\chi_{_{\varpi,\psi}},s,\psi')$. It is a higher dimensional metaplectic analog to Shahidi local coefficients. Precisely, if $n=1$ then $D(\chi_{_{\varpi,\psi}},s,\psi')=\tau(0,0,\chi,s,\psi')$ is the inverse of the local coefficient defined by Shahidi in \cite{Sha 1}. For the metaplectic $n=2$  case see \cite{Sz 2}.
\subsection{Reduction to an integral}
\begin{lem}  If $\ab \chi\bigl(\varpi^d) \ab <q^{Re(s)d}$ then
\begin{equation} \label{local as int again }\tau(i,j,\chi_{_{\varpi,\psi}},s,\psi')=q^{j-i} (\chi(\varpi)q^{-s})^{-i-j} \times$$ $$ \lim_{r \rightarrow \infty} \int_{F^*_{d,i+j}\cap \Pf^{-r}}  \ab z \ab^s \chi(z)\eta_\varpi(z)^{i-j}\xi_{\varpi,\psi}\bigl(\varpi^{-i-j}z\bigr)\psi'(z)  \, d_{\psi'}^*z .\end{equation}

\end{lem}
\begin{proof} Define
 $$K_m=\{g\in \slt \mid g=I_{2n} \, ( \operatorname{mod} \, \Pf^m) \}.$$
$\widetilde{K_m}$ is an open subgroup of $\mslt$. For $m$ sufficiently large, $\mslt$ splits over $K_m$ via the trivial section.  Denote
$K_m^+=\widetilde{K_m}\cap s\bigl(N(F)\bigr)$. Suppose that $m>\max\{e(\chi), e(\psi')\}$. Let  $f_{m,j,\chi_{_{\varpi,\psi}},s} \in I(\chi_{\varpi,\psi},s\bigr)$ be the function supported in $$\widetilde{B_d(F)}s\bigl( h(\varpi^j) K_m w\bigr)=\widetilde{B_d(F)} s\bigl(h(\varpi^j)\bigr)s(w) K^+_m $$ normalized such that
$$f_{m,j,\chi_{_{\varpi,\psi}},s}(bs\bigl(h(\varpi^j)\bigr)s(w)k)=\operatorname{Vol}^{-1}_{\psi'}(\Pf^m)(\chi_{\varpi,\psi},s\bigr) \cdot \delta^\half(b)$$ for all
$b \in \pbm, \, k \in K_m^+$.
Arguing as in Lemma 1.31 of \cite{KP} we have
$$\lambda_{i,\psi',\chi_{\varpi,\psi},s}\bigl( f_{m,j,\chi_{_{\varpi,\psi}},s} \bigr)=\begin{cases} 1 & i=j ;\\ 0 & i \neq j . \end{cases}$$
Thus, it is sufficient to show that if $\ab \chi\bigl(\varpi^d) \ab <q^{Re(s)d}$ then
\begin{equation} \label{this is what we need} \lambda^w_{i,\psi',\chi_{\varpi,\psi},s}\bigl( f_{m,j,\chi_{_{\varpi,\psi}},s} \bigr)=q^{j-i} (\chi(\varpi)q^{-s})^{-i-j} \times$$ $$ \lim_{r \rightarrow \infty} \int_{F^*_{d,i+j}\cap \Pf^{-r}}  \ab z \ab^s \chi(z)\eta_\varpi(z)^{i-j}\xi_{\varpi,\psi}\bigl(\varpi^{-i-j}z\bigr)\psi'(z)  \, d_{\psi'}^*z .\end{equation}
We have
$$\lambda^w_{i,\psi',\chi_{\varpi,\psi},s}\bigl( f_{m,j,\chi_{_{\varpi,\psi}},s} \bigr)=$$ $$\lim_{r \rightarrow \infty} \int_{\Pf^{-r}}\bigl((A_{w}\bigl(\chi_{\varpi,\psi},s\bigr)\bigl( f_{m,j,\chi_{_{\varpi,\psi}},s} \bigr)\bigl(s\bigl(h(\varpi^i)\bigr)s(w)^{-1}s(n(x))\bigr) \psi'^{-1}(x) \, d_{\psi'}x=$$
\begin{equation} \label{the double integral} \lim_{r \rightarrow \infty} \int_{\Pf^{-r}} \Bigl(\int_{F^*} f_{m,j,\chi_{_{\varpi,\psi}},s}\bigl(s(w)s(n(y))s\bigl(h(\varpi^i)\bigr)s(w)s(n(x))\bigr)  \, d_{\psi'}y \Bigr) \psi'^{-1}(x) \, d_{\psi'}x.  \end{equation}
By a matrix multiplication and by using the cocycle formula \eqref{rao} we have
$$s(w)s \bigl(n(y)\bigr)s\bigl(h(\varpi^i)\bigr)s(w)s\bigl(n(x)\bigr)=\left(\begin{pmatrix} _{-\varpi^i y^{-1}} & _{\varpi^{-i}}\\_{0} &
_{-\varpi^{-i}y}\end{pmatrix}wn(x-\varpi^{2i}y^{-1}),(\varpi,-1)^i(-\varpi^i,y) \right).$$
Hence, we may write the inner integral in \eqref{the double integral} as
$$\int_{F^*} \ab y \ab f_{m,j,\chi_{_{\varpi,\psi}},s}\left(\begin{pmatrix} _{-\varpi^i y^{-1}} & _{\varpi^{-i}}\\_{0} &
_{-\varpi^{-i}y}\end{pmatrix}wn(x-\varpi^{2i}y^{-1}), (\varpi,-1)^i(-\varpi^i,y) \right)  \, d_{\psi'}^*y.$$
Making the change of variables $z=-\varpi^{2i}y^{-1}$ we obtain
$$ q^{-2i} \int_{F^*} \ab z \ab^{-1} f_{m,j,\chi_{_{\varpi,\psi}},s}\left(\begin{pmatrix} _{\varpi^{-i}z} & _{-\varpi^{-i}}\\_{0} &
_{z^{-1}\varpi^{i}}\end{pmatrix}wn(x+z),\eta^i_\varpi(z)(-1,z) \right)  \, d_{\psi'}^*z. $$
Observe now that unless $x+z \in \Pf^m$ the last integrand vanishes. On the other hand ,if $x+z \in \Pf^m$ then $\psi'^{-1}(x)=\psi'(z)$. By the right invariance property of  $f_{m,j,\chi_{_{\varpi,\psi}},s}$ we have

$$\lambda^w_{i,\psi',\chi_{\varpi,\psi},s}\bigl( f_{m,j,\chi_{_{\varpi,\psi}},s} \bigr)=$$ $$ q^{-2i} \, \lim_{r \rightarrow \infty} \int_{\Pf^{-r}} \left( \int_{z \in F^*, \, x-\varpi^i z \in \Pf^m}  \ab z \ab^{-1} f_{m,j,\chi_{_{\varpi,\psi}},s}\left(\begin{pmatrix} _{\varpi^{-i}z} & _{-\varpi^{-i}}\\_{0} &
_{z^{-1}\varpi^i}\end{pmatrix}w,\eta^i_\varpi(z) (-1,z) \right)   \psi'(z) \, d_{\psi'}^*z \right) \, d_{\psi'}x.$$
Changing the order of integration gives
$$q^{-2i}  \, \lim_{r \rightarrow \infty}  \int_{z \in F^*}  \ab z \ab^{-1} f_{m,j,\chi_{_{\varpi,\psi}},s} \left(\begin{pmatrix} _{\varpi^{-i}z} & _{-a^{-i}}\\_{0} &_{z^{-1}\varpi^i}\end{pmatrix}w,\eta^i_\varpi(z) (-1,z)\right)   \psi'(z)  \phi(z,r,m) \, d_{\psi'}^*z$$
where
$$\phi(z,r,m)=\operatorname{Vol}_{\psi'} \bigl(\Pf^{-r} \cap z+\Pf^m \bigr).$$
Thus,
$$\lambda^w_{i,\psi',\chi_{\varpi,\psi},s}\bigl( f_{m,j,\chi_{_{\varpi,\psi}},s} \bigr)=$$ $$q^{-2i}\operatorname{Vol}_{\psi'} \bigl(\Pf^m \bigr) \lim_{r \rightarrow \infty} \int_{F^* \cap \Pf^{-r}} \ab z \ab^{-1} f_{m,j,\chi_{_{\varpi,\psi}},s}\left(\begin{pmatrix} _{\varpi^{-i}z} & _{-\varpi^{-i}}\\_{0} &
_{z^{-1}\varpi^i}\end{pmatrix}w,\eta^i_\varpi(z) (-1,z) \right)   \psi'(z)  \, d_{\psi'}^*z.$$
We now write
$$\left(\begin{pmatrix} _{\varpi^{-i}z} & _{-\varpi^{-i}}\\_{0} &
_{z^{-1}\varpi^i}\end{pmatrix}w,\eta^i_\varpi(z) (-1,z)\right)=\left(\begin{pmatrix} _{(\varpi)^{-i-j}z} & _{-\varpi^{j-i}}\\_{0} &
_{z^{-1}\varpi^{i+j}}\end{pmatrix},\eta^{i-j}(z)\right) s\left(h(\varpi^j)\right) s(w).$$
Recalling the definition of $f_{m,j,\chi_{_{\varpi,\psi}},s}$, \eqref{this is what we need} now follows.
\end{proof}
Note that the factor $q^{j-i}$ in \eqref{local as int again } is eliminated  if we redefine $\lambda_{i,\psi',\cdot,\cdot}$ by multiplying it
by $\delta^{\half}h(\varpi^i)=q^{-i}$. We shall not do so since we want our coefficients to be compatible with those defined in \cite{KP}.

\subsection{Explicit Formulas} \label{Explicit Formulas}
We shall now assume, without loss of generality, that $\psi$ is spherical. Following remark \ref{remark on split} we shall also assume and that if $n=0 \, (\operatorname{mod }4)$ then $\gamma_\psi(\varpi)=1$. We shall now give formulas for $D(\chi_{_{\varpi,\psi}},s,\psi')$ under the assumption that $\psi'=\psi$. This last assumption does affect the analytic behaviour of the local coefficients. However, the computation below is sufficient for our purposes, i.e., the computation of the Plancherel measure in Theorem \ref{main formula} below. Moreover, our computation can easily be modified to include all Whittaker characters. Recall that  once $\psi$ and $\varpi$ have been fixed, the actual inducing data for $I\bigl(\chi_{\varpi,\psi},s\bigr)$ is $s \, \bigl(\operatorname{mod} \frac{2\varpi i}{d \ln(q)} \Z \bigr)$ and the restriction of $\chi$ to ${F^*}^d$. Using Lemma \ref{cartan rep} we may assume that if $n$ is odd then $\chi$ is either trivial or that $\chi^{n}$ is ramified. For even $n$ we have to take another case into consideration, i.e., $\chi=\eta_\varpi$. Last, we formulate our results using $\epsilon$ factors and $L-$functions. Thus, it is convenient to think of $\chi$ as a character of $F^*$ rather than a character of $F^*_d$. Of course, our formulas for $\tau(i,j,\chi_{_{\varpi,\psi}},s,\psi)$ depends only on the restriction of $\chi$ to $F^*_d$.

For an integer $k$ define $k'$ to be the unique number such that $$k'=k \, ( \operatorname{mod} \, n), \, 0 \leq k \leq n-1.$$

\begin{lem} \label{explicit odd} Suppose that $n$ is odd. We can omit the subscript ${\varpi,\psi}$ from $\chi_{_{\varpi,\psi}}$\\
1. We have
$$\tau(j,j,1,s,\psi)= L(ns,1) \times \begin{cases} (1-q^{-1})  & 2j<n-1;   \\ \\ L^{-1}(1-ns,1) & 2j=n-1;  \\ \\ q^{ns}(1-q^{-1})  & 2j>n-1.  \end{cases}$$

and for $i \neq j$ we have
$$\tau(i,j,1,s,\psi)= \begin{cases} q^{j-i+s(n-1)}\epsilon^{-1} \bigl(s,\eta_\varpi^{i-j},\psi \bigr)  & j+i=n-1 ; \\ \\ 0  & otherwise . \end{cases}$$

2. If $\chi^n$ is ramified then
$$\tau(i,j,\chi,s,\psi)= \begin{cases} \chi(-1)q^{j-i} \bigl(\chi(\varpi)q^{-s}\bigr)^{-i-j}\epsilon^{-1} \bigl(s,\chi \eta_\varpi^{i-j},\psi \bigr)  & j+i+e(\chi)=0 \,  \operatorname{mod}(n) ; \\ \\ 0  & otherwise . \end{cases}$$
\end{lem}
\begin{proof} Define $\phi=q^{-r}\1_{1+\Pf^r}$. A standard computation shows that $\widehat{\phi}(x)=\psi(x)1_{\Pf^{-r}}(x)$. Thus,
$$ \int_{F^*_{d,i+j} \cap \Pf^{-r}}  \ab z \ab^s \chi(z)\eta_\varpi^{i-j}(z)\psi'(z)  \, d^*z=\zeta_{n,k}(s,\chi\eta_\varpi^{i-j},\widehat{\phi}).$$
It is easy to see that if $r>\max \{1,e(\chi)\}$ then
$$\zeta_{n,k}(1-s,\chi^{-1}\eta_\varpi^{j-i},\phi)=\begin{cases} 1 & k=0 \, ( \operatorname{mod} \, n) ;\\ \\ 0  & otherwise.  \end{cases}$$
Therefore, by Theorem \ref{aritur} we have
$$\tau(i,j,\chi,s,\psi)=q^{j-i} (\chi(\varpi)q^{-s})^{-i-j}\theta_{\bigl(i+j+e(\eta_\varpi^{i-j}\chi)\bigr)'}(s,\chi\eta_\varpi^{i-j},\psi).$$

The proof is now completed by a straight forward case by case computation.
\end{proof}
For $n \in 2\Z$  define
$$\alpha(n)=\begin{cases} 1 & n=0 \, ( \operatorname{mod} \, 4) ;\\ \\ 0  & n=2 \, ( \operatorname{mod} \, 4).  \end{cases}$$
\begin{lem} \label{explicit even} Suppose that $n$ is even\\
1. We have
$$\tau(j,j,1_{\varpi,\psi},s,\psi)=\begin{cases} \frac {L(ns,1) L(\frac {1-ns}{2},1)} {L(1-ns,1)L(\frac{1+ns}{2},1)}&  \, j=\frac{d-1}{2} ;\\ \\  (1-q^{-1}) L(ns,1)  & otherwise  \end{cases},$$
and for $i \neq j$ we have
$$\tau(i,j,1_{\varpi,\psi},s,\psi)= \begin{cases} q^{j-i+s(d-1)}\epsilon^{-1}(2s,\eta_\varpi^{2(i-j)},{\psi_{_2}}) \epsilon(s+\half,\eta_\varpi^{i-j},\psi)
  & j+i=d-1 ; \\ \\ 0  & otherwise. \end{cases}$$
2. If $j-i=1 \,  \operatorname{mod}(d),$  then

$$\tau(i,j,({\eta_\varpi})_{\varpi,\psi},s,\psi)=q^{(s+1)(j-i)}\epsilon(s+\half,\eta^d_\varpi,\psi)L(1,ns)  \begin{cases} L^{-1}(1-ns,1) & (i,j)=(d-1,0)   ; \\ \\
(1-q^{-1})  & otherwise   \end{cases},$$
and if $j-i \neq1 ,\,  \operatorname{mod}(d)$ then
$$\tau(i,j,({\eta_\varpi})_{\varpi,\psi},s,\psi)= \begin{cases} q^{(j-i)+s(d-1)} \epsilon(s+\half, \eta^{i-j+1}_\varpi,\psi)\epsilon^{-1}(2s,\eta^{2(i-j+1)}_\varpi,\psi_2)& i+j=d-1  ;  \\ \\
0  & otherwise  . \end{cases}$$
3. Suppose that $\chi^n$ is ramified. If $j+i+e(\chi)=0 \, ( \operatorname{mod} \, d)$,  then

$$\tau(i,j,\chi_{\varpi,\psi},s,\psi)=$$ $$\ q^{j-i} (\chi(\varpi)q^{-s})^{-i-j}   \chi(-1)\epsilon^{-1}(2s,\chi^2\eta_\varpi^{2(i-j)},{\psi_{_2}}) \epsilon(s+\half,\chi\eta_\varpi^{(d+1)(i-j)+\alpha(n)(e(\chi)+i+j)},\psi), $$
and otherwise $\tau(i,j,\chi_{\varpi,\psi},s,\psi)=0$.

\end{lem}
\begin{proof} By Lemma \ref{split porp} we have

$$\tau(i,j,\chi_{\varpi,\psi},s,\psi)=q^{j-i} (\chi(\varpi)q^{-s})^{-i-j} \times $$
$$\Bigl( \int_{F^*_{n,i+j}}  \ab z \ab^s \chi(z)\eta_\varpi(z)^{(d+1)(i-j)}\gamma_\psi\bigl(z\bigr)\psi'(z)  \, d^*z+\int_{F^*_{n,d+i+j}}  \ab z \ab^s \chi(z)\eta_\varpi(z)^{(d+1)(i-j)+d\alpha(n)}\gamma_\psi\bigl(z\bigr)\psi'(z)  \, d^*z \Bigr).$$
Define again $\phi=q^{-r}\1_{1+\Pf^r}$. It was shown in Section 2 of \cite{Sz 3} that  $$\widetilde{\phi}(x)=\gamma_\psi^{-1}(x)\psi(x)\1_{{\Pf^{-r}}}(x).$$
Utilizing Theorem  \ref{meta ari} and arguing as in the proof of Theorem \ref{explicit odd} we obtain
$$\tau(i,j,\chi_{_{\psi}},s,\psi)=q^{j-i} (\chi(\varpi)q^{-s})^{-i-j} \times$$  \begin{equation} \label{to be reminded soon} \Bigl(\widetilde{\theta}_{\bigl(e(\chi^2 \eta_\varpi^{2(i-j)})+i+j\bigr)'}(s,\chi\eta_\varpi^{(d+1)(i-j)},\psi)+\widetilde{\theta}_{\bigl(e(\chi^2 \eta_\varpi^{2(i-j)})+d+i+j\bigr)'}(s,\chi\eta_\varpi^{(d+1)(i-j)+d\alpha(n)},\psi) \Bigr). \end{equation}
The proof is now reduced to a straight forward case by case computation.

\end{proof}
\section {An Irreducibility result.} \label{irr res}
Fix $\tau$, a genuine smooth admissible irreducible representation of $\widetilde{H(F)}$. By Lemma \ref{cartan rep} $\tau \simeq i\bigl(\chi_{\varpi,\psi} \bigr)$ for some character $\chi$ of $F^*_d$.
Extend $\chi$ to a character of $F^*$. Using Lemma \ref{cartan rep} again we observe that $\chi^n$ is independent of the particular chosen extension and of the splitting $\xi_{\varpi,\psi}$. Thus, it is an invariant of $\tau$. Define now $$I(\tau,s)=\operatorname{Ind}_{\widetilde{B(F)}}^{\mslt} \tau_s.$$
Let $$A_{w}(\tau,s ):I(\tau,s ) \rightarrow I(\tau^w,-s)$$ be the meromorphic continuation of the integral
\begin{equation} \label{noreal} \int_F f\bigl(s(w)s(n(x))g)\, d_\psi x. \end{equation}
Here $\tau^w$ is the representation $h \mapsto \tau\bigl(s(w)h s(w)^{-1} \bigr)$. For all but finitely many values of $q^s$, $A_{w}(\tau,s )$ is analytic and
$$\operatorname{Hom}_{\mslt} \Bigl(I(\tau,s ),I(\tau,s) \Bigr)$$
is one dimensional. Thus, there exists a rational function in $q^{-s}$, $\mu\bigl(\tau,s\bigr)$, such that
$$A_{w^{-1}}\bigl(I(\tau^w,-s)\bigr) \circ A_{w}\bigl(I(\tau,s )\bigr)=\mu^{-1}(\tau,s) \operatorname{Id}.$$
\begin{thm} \label{main formula}
$$\mu^{-1}(\tau,s\bigr)=q^{e(\chi^n,\psi)}\frac{L \bigl(ns,\chi^n \bigr)L \bigl(-ns,\chi^{-n}\bigr)}{L \bigl(1-ns,\chi^{-n} \bigr)L \bigl(1+ns,\chi^{n}\bigr)}.$$
\end{thm}
\begin{proof} Fixing isomorphisms $I(\tau,s ) \simeq I\bigl(\chi_{\varpi,\psi},s  \bigr)$, $I(\tau^w,-s) \simeq I \bigl(\chi^{-1}_{\varpi,\psi},-s  \bigr)$ one immediately sees that
$$A_{w^{-1}}\bigl(\chi^{-1}_{\varpi,\psi},-s\bigr)\circ A_{w}\bigl(\chi_{\varpi,\psi},s\bigr)=\mu^{-1}\bigl(\tau,s\bigr)\operatorname{Id}$$
(we choose the Haar measure to be same one used in \eqref{noreal}). Since $d_\psi x=q^{ \frac {e(\psi)}{2}} \, dx$ it is sufficient to prove this theorem assuming that $\psi$ is spherical. We also assume, without loss of generality, that $\gamma_\psi(\varpi)=1$ and that $\chi$ is one of the characters that appears in Lemmas \ref{explicit odd} and \ref{explicit even}. Next we note that
since $$A_{w^{-1}}\bigl(\chi^{-1}_{\varpi,\psi},-s\bigr)=\chi(-1)A_{w}\bigl(\chi^{-1}_{\varpi,\psi},-s\bigr),$$
the proof of the theorem is done once we show
$$D(\chi_{_{\varpi,\psi}},s,\psi)D(\chi^{-1}_{_{\varpi,\psi}},-s,\psi)=\chi(-1)q^{-e(\chi^n)}\frac{L \bigl(ns,\chi^n \bigr)L \bigl(-ns,\chi^{-n}\bigr)}{L \bigl(1-ns,\chi^{-n} \bigr)L \bigl(1+ns,\chi^{n}\bigr)}\operatorname{Id}.$$
This identity is proven by a case by case matrix multiplication using the formulas in Lemmas  \ref{explicit odd} and \ref{explicit even}. More precisely, these two lemmas give the formula for $D(\chi^{-1}_{_{\varpi,\psi}},-s,\psi)$ for all $\chi$ in discussion except in the case where $n$ is even and $\chi=\eta_\varpi$. The formula in this case follows directly from \eqref{to be reminded soon} as well. During the course of the computation one should use \eqref{Tate gamma}, \eqref{epsilon twist} and \eqref{epsilon product} along with the fact that for all the cases, except for the case where $n$ is even and $\chi=\eta_\varpi$, we have $e(\chi^n)=e(\chi)$. This last assertion is proven by a similar argument to the one we used in the proof of Theorem \ref{meta ari} utilizing the fact that $1+\Pf \in {F^*}^n$.
\end{proof}
The Knapp-Stein Dimension Theorem for quasisplit p-adic groups, \cite{Sil}, gives a reducibility criteria for parabolic induction in terms of the Plancherel measure. The proof of the analogous result for metaplectic groups should follow using similar arguments to those that appear in  \cite{Sil}. Although at this point there is no written proof it is rather standard to assume this result for covering groups. See \cite{Bud} for an  example involving an $n$ fold cover of $GL_2(F)$. Moreover, in \cite{Li} Li assumed the theory of R-groups for $\mspn$, the metaplectic double cover of $\spn$ which is based on the Knapp-Stein Dimension Theorem. Assuming the Knapp-Stein Dimension Theorem,  it was proven in \cite{Sz 4} that all genuine principal representations of $\mspn$ induced from a unitary character are irreducible. This result was proved unconditionally in \cite{HMa}. It can be easily shown that the irreducibility results for principal series representations of the metaplectic double cover of $\gspn$ given in \cite{Sz 5} and \cite{Sz 6} also agree with the Knapp-Stein Dimension Theorem . The results in \cite{Sz 5} and \cite{Sz 6} are also unconditional since these are based on \cite{Sz 4}. In the case at hand the Knapp-Stein Dimension Theorem is reduced to the following statement:
\begin{prop} Suppose that $\tau$ is unitary. Then, $I(\tau,0)$ is reducible if and only if $\tau \simeq \tau^w$ and $\mu^{-1}(\tau,s)$ is analytic at $s=0$.
\end{prop}
This proposition is contained in \cite{SA}, where the Knapp-Stein Dimension Theorem is proven for unitary parabolic induction from $P$ to $G$ and from $\widetilde{P}$ to $\widetilde{G}$ where $G$ is a reductive group defined over $F$, $P$ is a maximal parabolic subgroup of $G$,  $\widetilde{G}$ is a central extension of $G$ by $\mu_n$ and $\widetilde{P}$ is the preimage of $P$ in  $\widetilde{G}$.

\begin{thm}  \label{savin}Suppose that $\tau$ is unitary. Then $I(\tau,0)$ is reducible if and only if $n$ is odd and the projection of the central character of $\tau$ to $H^d(F)$ is a non-trivial quadratic character.
\end{thm}
\begin{proof} We first note that if $\tau \simeq i\bigl(\chi_{\varpi,\psi} \bigr)$ then $\tau^w \simeq i\bigl(\chi^{-1}_{\varpi,\psi} \bigr)$. Thus, by Lemma \ref{cartan rep} $\tau \simeq\tau^w$ if and only if $\chi=\chi^{-1}\eta_\varpi^{2m}$ for some $m=1,2, \ldots,d-1$. If we extend $\chi$ to $\F^*$ then the last equality is equivalent to $\chi^{2d}=1$. Suppose first that $n$ is even. In this case we have just shown that $\tau \simeq \tau^w$ implies that $\chi^n$ is trivial. By Theorem \ref{main formula} we occlude that $\mu^{-1}(\tau,s)$ has a pole at $s=0$. Suppose now that $n$ is odd and that  $\tau \simeq \tau^w$. In this case $\chi^{2n}$ is trivial. Thus, if $\chi^n$ is not trivial then $\mu^{-1}(\tau,s)$ is analytic at $s=0$. The assertions $\chi^{2n}=1$, \, $\chi^n \neq 1$ are equivalent to the assertion that the projection of the central character of $\tau$ to $H^d(F)$ is a non-trivial quadratic character.
\end{proof}
\section{Appendix: A result of W. Jay Sweet.}
In this appendix we shall not assume any restriction on the $p$-adic field $F$. Let $\chi$ be a character of  $\F^*$ and let $\psi$ be a non-trivial character of $\F$. We intend to prove here that for $Re(s)>>0$
\begin{equation} \label{J.Sweet} \lim_{r \rightarrow \infty}\int_{\Pf^{-r}}\chi(x) \ab x \ab^s \gamma_\psi(x)^{-1} \psi(x) \, d_\psi^* x=\gamma_F^{-1}(\psi_{-1}) \chi(-1)\gamma^{-1}(2s,\chi^{2},{\psi_{_2}}) \gamma(s+\half,\chi,\psi). \end{equation}
Our proof follows closely the original proof of Sweet in \cite{Sweet} although we did not keep all of  his notation.
\begin{lem} Suppose that \eqref{J.Sweet} holds  provided that $\psi$ is unramified. Then, it holds for all non-trivial $\psi$.
\end{lem}
\begin{proof} Assume that $e(\psi)=n$. Then $\psi'=\psi_{(\varpi^n)}$ is unramified and $d_\psi x=q^{\frac n 2}d_{\psi'} x=q^{\frac n 2}dx$. By definition of the Weil index we have
$$\gamma_\psi(x)=\gamma_{\psi'}(x)(x,\varpi^n)_2.$$
This implies that
$$\lim_{r \rightarrow \infty}\int_{\Pf^{-r}}\gamma_\psi^{-1}(x)\chi(x) \ab x \ab^s \psi(x) \, d_\psi^*x=q^{\frac n 2}\lim_{r \rightarrow \infty}\int_{\Pf^{-r}}\gamma_{\psi'}^{-1}(x)(x,\varpi^n)_2\chi(x) \ab x \ab^s \psi'(x\varpi^{-n}) \, d^*x.$$
We make the change integration variables $a=x \varpi^{-n}$ and obtain
$$\lim_{r \rightarrow \infty}\int_{\Pf^{-r}}\gamma_\psi^{-1}(x)\chi(x) \ab x \ab^s \psi(x) \, d_\psi^*x=q^{\frac n 2 -ns}\gamma_{\psi'}^{-1}(\varpi^n)(-1,\varpi^n)\chi(\varpi^n)\lim_{r \rightarrow \infty}\int_{\Pf^{-r}}\gamma_{\psi'}^{-1}(a)\chi(a) \ab a \ab^s \psi'(a) \, d^*a.$$
Since $\psi'$ is unramified it follows from our assumption that
$$\lim_{r \rightarrow \infty}\int_{\Pf^{-r}}\gamma_\psi^{-1}(x)\chi(x) \ab x \ab^s \psi(x) \, d_\psi^*x=$$ $$q^{\frac n 2 -ns}\gamma_{\psi'}^{-1}(\varpi^n)(-1,\varpi^n)\chi(\varpi^n) \gamma_F^{-1}(\psi'_{-1})\chi(-1)\gamma^{-1}(2s,\chi^{2},\psi'_2) \gamma(s+\half,\chi,\psi').$$
By \eqref{epsilon change psi} we have $$\lim_{r \rightarrow \infty}\int_{\Pf^{-r}}\gamma_\psi^{-1}(x)\chi(x) \ab x \ab^s \psi(x) \, d_\psi^*x=$$ $$\gamma_{\psi'}^{-1}(\varpi^n)(-1,\varpi^n) \gamma_F^{-1}(\psi_{-1}')\chi(-1)\gamma^{-1}(2s,\chi^{2},{\psi_{_2}}) \gamma(s+\half,\chi,\psi).$$
It is left to show that
$$\gamma_{\psi'}(\varpi^n)(-1,\varpi^n) \gamma_F(\psi_{-1}')=\gamma_F(\psi_{-1}).$$
Indeed, $$\gamma_{\psi'}(\varpi^n)(-1,\varpi^n) \gamma_F(\psi_{-1}')=\gamma_{\psi'}(-\varpi^{-n})\gamma_{\psi'}^{-1}(-1) \gamma_F(\psi_{-1}')=\gamma_F(\psi'_{(-\varpi^{-n})})=\gamma_F(\psi_{-1}).$$
\end{proof}

We shall assume from this point that $\psi$ is normalized.
Define $$e(2,\F)=log_q[\Of:2\Of]=-log_q \ab 2 \ab$$
and define $$c_\psi(a)= \lim_{r \rightarrow \infty}\int_{\Pf^{-r}} \psi(ax^2) \, dx.$$ With this
notation:
$$ \gamma_\psi(a)=\ab a \ab ^\half c_\psi(a)c^{-1}_\psi(1).$$
Proposition 3.3 of \cite{CC} states that
\begin{equation} \label{Weil id} \gamma_\F(a\psi)=\ab 2a \ab^{\half}c_\psi(a).\end{equation}
We fix an integer $M$ such that $M \geq \max\{2e+1,m(\chi)\}$. By Hensel's lemma, $1+\Pf^{2e+1} \subseteq {\F^*}^2$. This implies (Proposition 3.1 of \cite{CC}) that for $\ab a \ab \geq q^{-M}$ we have
\begin{equation} \label{CC identity} c_\psi(a)=\int_{\Pf^{-M}} \psi(ax^2)\, dx. \end{equation}
Note that since $\gamma_\psi \in \C^{1}$, \eqref{CC identity} implies
\begin{equation} \label{from CC identity} \gamma_\psi^{-1}=\gamma_{(\psi_{-1})}. \end{equation}
The following is Lemma 2.2 of \cite{Sweet}. We give an elementary proof.
\begin{lem} \label{sweet lemma}  For $z \in \F^*$ we have
$$\int_{\Pf^{-M}} \psi^{-1}(zc^2-2c) \, dc=\psi(z^{-1})\ab z \ab ^{-\half}\gamma_\psi^{-1}(z)c_\psi(-1)1_{\Pf^{-M}}(z^{-1}).$$
\end{lem}
\begin{proof} Write $x=c-z^{-1}.$ This gives
\begin{equation} \label{var c}\int_{\Pf^{-M}} \psi^{-1}(zc^2-2c) \, dc=\psi(z^{-1}) \int _{-z^{-1}+\Pf^{-M}} \psi(-zx^2) \, dx.\end{equation}
Assume first that $\ab z \ab  \geq q^{-M}$. In this case  $-z^{-1} \in \Pf^{-M}$. Thus, by \eqref{CC identity} and \eqref{from CC identity} we have
$$\int_{\Pf^{-M}} \psi^{-1}(zc^2-2c)\, dc=\psi(z^{-1}) \int _{\Pf^{-M}} \psi(-zx^2) \, dx=c_\psi(-z)=\ab z \ab ^{-\half}\gamma_\psi^{-1}(z)c_\psi(-1).$$
The Lemma is now proven for this case. Suppose now  that $\ab z \ab  < q^{-M}$. We shall write $z=u_{_0}\varpi^k$, where $u \in \Of^*$, $k>M$. Note that in this case $$-z^{-1}+\Pf^{-M}=-z^{-1}(1+\Pf^{k-M}).$$ Thus, by  changing $x=-z^{-1}y$ in the right hand side of \eqref{var c}, we only need to show that

\begin{equation} \label{show this} \int _{1+\Pf^{k-M}} \psi(-z^{-1}y^2) \, dy=0.\end{equation}
Suppose that  $2M > k >M$. In this case we can pick integer $t$ such that
$$\max \{2k-2M,2e+1\} \leq t <k.$$
For any $u\in \Of$, we have $1+u\varpi^t \in {F^*}^2$. We claim that we can always find a solution for the equation $x^2=1+u\varpi^t$ such that $x \in 1+\Pf^{k-M}$. Indeed, if
$$\ab x^2-1 \ab \leq q^{-t} \leq q^{-(2k-2M)},$$
then since $\ab -x-1 \ab \cdot \ab x-1 \ab =\ab x^2 -1 \ab$ it is not possible that both $\ab -x-1 \ab$ and $\ab x-1 \ab$ are greater then $q^{-(k-M)}$. We shall denote this solution by $\sqrt{1+u\varpi^t}$. It follows that for any $u \in \Of$ we can change $y\mapsto y\sqrt{1+u\varpi^t}$ in the left hand side of \eqref{show this}, without changing the measure or the domain of integration. This gives

$$\int _{1+\Pf^{k-M}} \psi(-z^{-1}y^2) \, dy=\int_{\Of} \int_{1+\Pf^{k-M}} \psi_{-u_0^{-1}}\bigl(\varpi^{-k}y^2(1+u\varpi^t)\bigr) \, dy \, du=$$
$$\int_{1+\Pf^{k-M}} \psi_{-u_0}(\varpi^{k} y^2) \Bigl( \int_{\Of} \psi_{-y^2u_0^{-1}}(\varpi^{t-k}u) \, du \Bigr) \, dy.$$
Since $k>t$, the last inner integral vanishes. We now deal with the case $k \geq 2M$. We change $y=1+u \varpi^{k-M}$ in the left hand side of \eqref{show this}. We have to show that
$$\int _{\Of} \psi_{-u_{_0}^{-1}}(2\varpi^{-M}u) \psi_{-u_{_0}^{-1}}(\varpi^{k-2M}u^2) \, du=0.$$
This follows from the fact that since $k \geq 2M$, $\psi_{-u_{_0}^{-1}}(\varpi^{k-2M}u^2)=1$ for all $u \in \Of$ and from the fact that the map $u \mapsto  \psi_{-u_{_0}^{-1}}(2\varpi^{-M}u)$ is a non-trivial character of $\Of$.
\end{proof}

In what follows, we shall have two different additive characters. To stress the dependence of the Fourier transom of $\phi \in S(F)$ on the additive character we shall write $\phi_\psi$ instead of $\widehat{\phi}$. While we take the $\psi$-self dual measure for the integration defining the $\psi$-Fourier transform we shall assume here that  the measure in the integral defining the Mellin transforms is $d^*x$. This last convenient assumption does not change the definition of the Tate $\gamma-$factor in \eqref{tate def}. Define \begin{equation} \label{f def} f^{^M}(x)=\psi(2x) \1_{\Pf^{-M}}(x).\end{equation} By a straight forward computation one sees that

\begin{equation} \label{f fou} f_{{\psi_{_2}}}^{^M}(x)=q^{M-\frac e 2} \1_{1+\Pf^{M-e}}(-x) \end{equation}
and that
\begin{equation} \label{f fou mel} \zeta(s,\chi^2,f_{{\psi_{_2}}}^{^M},)=q^{\frac e 2}=\ab 2 \ab ^{-\half}.\end{equation}
We now come to the heart of the proof given in \cite{Sweet}. Denote
$$I(M,\chi,\psi,s)=\int_{\Pf^{-M}}\gamma_\psi^{-1}(x)\chi(x) \ab x \ab^s \psi(x) \, d^*x.$$
\begin{thm} \label{sweet thm}$$I(M,\chi,\psi,s)=\gamma_F^{-1}(\psi_{-1})\chi(-1)\gamma^{-1}(2s,\chi^{2},{\psi_{_2}}) \gamma(s+\half,\chi,\psi).$$
\end{thm}

\begin{proof}
$$\int_{\Pf^{-M}}\gamma_\psi^{-1}(x)\chi(x) \ab x \ab^s \psi(x) \, d^*x=\int_{\F}\gamma_\psi^{-1}(x)\chi(x) \ab x \ab^s \psi(x)1_{\Pf^{-M}}(x) \, d^*x=[x \mapsto x^{-1}]=$$
$$\int_{\F}\gamma_\psi^{-1}(x)\chi^{-1}(x) \ab x \ab^{-s} \psi(x^{-1})1_{\Pf^{-M}}(x^{-1}) \, d^*x=$$
$$c^{-1}_\psi(-1)\int_{\F}\chi^{-1}(x) \ab x \ab^{\half-s} \Bigl( \ab x \ab^{-\half}\gamma_\psi^{-1}(x)c_\psi(-1) \psi(x^{-1})1_{\Pf^{-M}}(x^{-1}) \Bigr)\, d^*x.$$
By Lemma \ref{sweet lemma},
$$I(M,\chi,\psi,s)=c^{-1}_\psi(-1)\int_{\F}\chi^{-1}(x) \ab x \ab^{\half-s} \Bigl( \int_{\Pf^{-M}} \psi^{-1}(xc^2-2c) \, dc \Bigr) \, d^*x=[x \mapsto -x]=$$
$$\chi(-1)c^{-1}_\psi(-1)\int_{\F}\chi^{-1}(x) \ab x \ab^{\half-s} \Bigl( \int_{\Pf^{-M}} \psi(xc^2) \psi(2c) \, dc \Bigr) \, d^*x.$$
Recalling \eqref{f def} we obtain
\begin{equation} \label{inter} I(M,\chi,\psi,s)=c^{-1}_\psi(-1)\chi(-1)\int_{\F}\chi^{-1}(x) \ab x \ab^{\half-s} \Bigl( \int_{\F} \psi(xc^2) f^{^M}(2c) \, dc \Bigr) \, d^*x.\end{equation}
Let $\phi \in S(\F)$ be such that $\phi_\psi(0)=0$ and such that $\zeta(s,\chi,\phi)$ is not the zero function and define:
\begin{equation} \label{I tag}I'(M,\chi,\psi,s)=I(M,\chi,\psi,s)\zeta(s+\half,\chi,\phi)\zeta(1-2s,\chi^{-2},f_{{\psi_{_2}}}^{^M}). \end{equation}
By \eqref{Weil id}, \eqref{f fou mel} and \eqref{inter} we have
$$I'(M,\chi,\psi,s)=\gamma^{-1}_F(\psi_{-1})\chi(-1)\int_{\F^*} \phi(y) \chi(y) \ab y \ab ^{s+\half} \, d^* \! y
\int_{\F}\chi^{-1}(x) \ab x \ab^{\half-s} \Bigl( \int_{\F} \psi(xc^2) f^{^M}(2c) \, dc \Bigr) \, d^*x$$
(note that since $Re(s)>>0$, both integrals in the right hand side are absolutely convergent).
Recalling that $d^*y=\frac {dy}{\ab y \ab}$ we get
$$I'(M,\chi,\psi,s)=\gamma^{-1}_F(\psi_{-1})\chi(-1)\int_{\F^*}\int_{\F^*} \chi(yx^{-1})\phi(y)\ab xy^{-1} \ab ^{\half-s}
\Bigl( \int_{\F} \psi(xc^2) f^{^M}(c) \, dc \Bigr) \, d^*x \, dy=$$
$$ [x=yz]=\gamma^{-1}_F(\psi_{-1})\chi(-1)\int_{\F^*}\int_{\F^*} \chi^{-1}(z)\phi(y)\ab z \ab ^{\half-s}
\Bigl( \int_{\F} \psi(yzc^2) f^{^M}(c) \, dc \Bigr) \, d^*z \, dy.$$
We would like to change to $z-y$ order of integration. Therefore, we need to show that
$$G(y,z)=\chi^{-1}(z)\phi(y)\ab z \ab ^{-\half-s} \Bigl( \int_{\F} \psi(yzc^2) f^{^M}(c) \, dc \Bigr)$$ is integrable on $\F \times \F$ (with respect to $dy \, dz$.) Indeed, since $\phi \in S(F)$, $G(y,z)$ vanishes for large values of $\ab y \ab$. By Lemma \ref{sweet lemma}, $G(y,z)$ also vanishes when $\ab yz \ab<q^{-M}$. This implies that as $\ab z \ab \rightarrow \infty$ then $G(y,z)$ is bounded by $c\ab z \ab ^{-\half-s}$. We also change the $y-c$ order of integration. This change is easily justified since
$$K(c,y)=\phi(y) ^{-\half-s}  \psi(yzc^2) f^{^M}(c)$$
is a bounded compactly supported function in the $c-y$ plane. Thus,
$$I'(M,\chi,\psi,s)=\gamma^{-1}_F(\psi_{-1})\chi(-1)\int_{\F^*}\int_{\F^*} \chi^{-1}(z)f^{^M}(c)\ab z \ab ^{\half-s}
\Bigl( \int_{\F}\phi(y) \psi(yzc^2)  \, dy \Bigr)  \,dc \, d^*z=$$
$$\gamma^{-1}_F(\psi_{-1})\chi(-1)\int_{\F^*}\int_{\F^*} \chi^{-1}(z)f^{^M}(c)\ab z \ab ^{\half-s}
\phi_\psi(zc^2) \,dc \, d^*z.$$
We note that
$$F(c,z)=\chi^{-1}(z)f^{^M}(c)\ab z \ab ^{\half-s} \phi_\psi(zc^2)$$
is compactly supported on the $c-z$ plane, and since we assume that $\phi_\psi$ is supported away from 0, it follows that $F(c,z)$ is bounded.
This implies that we can change the $c-z$ order of integration. We then change $zc^2=t$ in the inner $d^*z$ integral and obtain:
$$I'(M,\chi,\psi,s)=\gamma^{-1}_F(\psi_{-1})\chi(-1)\int_{\F^*}\int_{\F^*} \chi^{-1}(tc^{-2})f^{^M}(c)\ab tc^{-2} \ab ^{\half-s}
\phi_\psi(t) \, d^*t \, dc= $$
$$\gamma_F^{-1}(\psi_{-1})\chi(-1)\int_{\F^*}\phi_\psi(t) \chi^{-1}(t) \ab t \ab^{\half-s} \, d^*t \int_{\F^*} f^{^M}(c) \chi^{2}(c) \ab c \ab^{2s} \, d^*c.$$
Recalling \eqref{I tag}, we have shown that
$$\gamma^{-1}_F(\psi_{-1})\chi(-1) \zeta(-s+\half,\chi^{-1},\phi_\psi,)\zeta(2s,\chi^2,f^{^M})=I(M,\chi,\psi,s)\zeta(s+\half,\chi,\phi)\zeta(1-2s,\chi^{-2},f_{{\psi_{_2}}}^{^M}).$$
By \eqref{tate def}, the functional equation defining the Tate $\gamma$-factor, the theorem now follows.
\end{proof}


\begin{thebibliography}{}

\bibitem{Ariturk} Ariturk, H. {\it On the composition series of principal series representations of a three-fold covering group of SL(2,K).} Nagoya Math. J. 77 (1980), pp. 177-–196.


\bibitem{Bud} Budden M., {\it Local coefficient matrices of metaplectic groups.}
J. Lie Theory 16, no. 2, (2006), pp. 239-249.

\bibitem{CC} Chais J., Cong X., {A note on Weil index.} Sci. China Ser. A 50 (2007), no. 7, pp. 951-–956.

\bibitem{KP} Kazhdan, D. A., Patterson, S. J,{\it Metaplectic forms},
 Inst. Hautes etudes Sci. Publ. Math. No. 59,(1984), pp. 35--142.

\bibitem{FGS} Friedberg, S., Goldberg, D., Szpruch, D., In preparation.

\bibitem{GanGao} Gan W.T, Gao F., {\it The Langlands-Weissman Program for Brylinski-Deligne extensions.} Preprint. Available at http://arxiv.org/abs/1409.4039


\bibitem{GP80} Gelbart S., Piatetski-Shapiro I., {\it Distinguished representations and modular forms of half-integral weight.} Invent. Math. 59 (1980), no. 2, pp. 145–-188.

\bibitem{GP81} Gelbart S., Piatetski-Shapiro I., {\it On Shimura's correspondence for modular forms of half-integral weight.} , pp. 1–-39, Tata Inst. Fund. Res. Studies in Math., 10, Tata Inst. Fundamental Res., Bombay, 1981.

\bibitem {HMa} Hanzer M., Matic I., {\it Irreducibility of the unitary principal
series of $p$-adic $\widetilde{Sp(n)}$.} Pacific Journal of Mathematics 248-1 (2010), pp. 107-137.

\bibitem{Kub} Kubota T., {\it Automorphic functions and the reciprocity law in a number ﬁeld.} Kinokuniya
book store (1969).

\bibitem{Lang}  Langlands, R.P., {\it On the functional equations
satisfied by Eisenstein series.} Lecture Notes in Mathematics,
Vol. 544. Springer-Verlag, Berlin-New York, 1976.

\bibitem{Li} Li W.W., {\it La formule des traces pour les revêtements de groupes réductifs connexes. II. Analyse harmonique locale.} Ann. Sci. Éc. Norm. Supér. (4) 45 (2012), no. 5, pp. 787-–859 (2013).
\bibitem{Mc}  McNamara P.J., {\it Principal Series Representations of Metaplectic Groups Over Local
Fields} in  {\it Multiple Dirichlet Series, L-functions and Automorphic Forms,} editors Bump
D., Friedberg S. and Goldfeld D. Birkhuser (2012).

\bibitem{Mc2}  McNamara P.J., {\it The metaplectic Casselman-Shalika formula}. To appear, Transactions of the American Mathematical Society. DOI: http://dx.doi.org/10.1090/tran/6597


\bibitem{SA} Savin G., {\it An addendum to Casselman's notes.} Preprint.

\bibitem{Sha 1} Shahidi F., {\it On certain $L$-functions.}
Amer. J. Math. 103 (1981), no. 2, pp. 297-355.

\bibitem{Sha 90} Shahidi, F., {\it Langlands' conjecture on
Plancherel measures for $p$-adic groups} in {\it Harmonic analysis
on reductive groups}, Birkhauser Boston, Boston, MA, 1991, pp
277-295.

\bibitem{Sil} Silberger A.J., {\it the Knapp-Stein dimension theorem
for p-adic groups}, Proceeding of the Ameraican mathematical
society, Vol. 68, No. 2, (1978), pp. 243-246 and Silberger A.J., {\it Correction: ''The Knapp-Stein dimension theorem for
 $p$-adic groups''}, Proceeding of the Ameraican mathematical
society, Vol. 76, No. 1, (1979), pp. 169-170.


\bibitem{Sweet} Sweet W. J., {Functional equations of p-adic zeta integrals and representations of the metaplectic group.} Preprint (1995).

\bibitem{Sz 2} Szpruch D., {\it Computation of the local coefficients for principal series representations of the metaplectic double
cover of $SL_2(\F)$.}  Journal of Number Theory, Vol. 129, (2009), pp.
2180-2213


\bibitem{Sz 3} Szpruch D., {\it On the existence of a p-adic metaplectic Tate-type $\widetilde{\gamma}$-factor.}
The Ramanujan Journal, Vol. 26 (2011), no. 1, pp 45-53.

\bibitem{Sz 4}  Szpruch D., {\it Some irreducibility theorems of parabolic induction on the metaplectic group via the Langlands-Shahidi method.} Israel Journal of Mathematics, Vol. 195 (2013), no. 2, pp. 897–-971

\bibitem{Sz 5} Szpruch D., {\it Some results in the theory of genuine representations of the metaplectic double cover of GSp(2,F) over p-adic fields.} Journal of Algebra, Vol. 388 (2013) pp. 160--193.

\bibitem{Sz 6} Szpruch D., {\it Symmetric Spherical Whittaker functions on the metaplectic double cover of Gsp(2n,F).} Canadian Journal of Mathematics Vol. 67 (2015), no. 1, pp. 214--240.

\bibitem{T} Tate J.T, {\it Fourier analysis in number fields, and Hecke's zeta-functions} in
{\it Algebraic Number Theory} (Proc. Instructional Conf.,
Brighton, 1965), 1967, pp. 305-347.

\bibitem{T2} Tate J.T, {\it Number theoretic background. Automorphic forms} in {\it representations and L-functions}. (Proc. Sympos. Pure Math., Oregon State Univ., Corvallis, Ore., 1977), Part 2, pp. 3–-26, Proc. Sympos. Pure Math., XXXIII, Amer. Math. Soc., Providence, R.I., 1979.


\bibitem{Weil} Weil A., {\it Sur certains groupes d'operateurs unitaires}, Acta Math Vol { 111}, 1964, pp.
143--211.

\bibitem{Weil book} Weil A. {\it Basic number theory.} Reprint of the second (1973) edition. Classics in Mathematics. Springer-Verlag, Berlin, 1995. xviii+315 pp.


\end{thebibliography}
\end{document}